\documentclass{amsart}     
\usepackage{graphicx, hyperref}  
\title{Casson towers and filtrations\\ of the smooth knot concordance group}

\author{Arunima Ray}
\address{Department of Mathematics, MS-050, Brandeis University, 415 South St., Waltham, MA 02453.}
\email{aruray@brandeis.edu}
\urladdr{http://people.brandeis.edu/~aruray}

\newtheorem{proposition}{Proposition}[section]

\newtheorem*{cor_33'}{Corollary 3.3$'$}
\newtheorem{theorem}[proposition]{Theorem}
\newtheorem*{thm_1}{Theorem 1}
\newtheorem*{thm_A}{Theorem A}
\newtheorem*{thm_A'}{Theorem A$'$}
\newtheorem*{thm_B}{Theorem B}
\newtheorem*{thm_B'}{Theorem B$'$}
\newtheorem{corollary}[proposition]{Corollary}
\newtheorem*{cor_1}{Corollary 1}
\newtheorem*{cor_2}{Corollary 2}

\theoremstyle{definition}
\newtheorem{definition}[proposition]{Definition}
\newtheorem{remark}[proposition]{Remark}
\newtheorem{example}[proposition]{Example}
\newtheorem*{defn_1}{Definition 1}
\newtheorem*{defn_2}{Definition 2}
\newtheorem*{defn_3}{Definition 3}
\newtheorem*{defn_4}{Definition 4}
\newtheorem*{defn_1'}{Definition 1$'$}
\newtheorem*{defn_2'}{Definition 2$'$}
\newtheorem*{defn_3'}{Definition 3$'$}
\newtheorem*{defn_4'}{Definition 4$'$}
\newcommand{\Wh}[1][\pm]{\text{Wh}^{#1}}
\newcommand{\CP}{\mathbb{CP}}
\newcommand{\Arf}{\text{Arf}}

\begin{document}

\begin{abstract}    
The \textit{$n$--solvable filtration} $\{\mathcal{F}_n\}_{n=0}^\infty$ of the smooth knot concordance group (denoted by $\mathcal{C}$), due to Cochran--Orr--Teichner, has been instrumental in the study of knot concordance in recent years. Part of its significance is due to the fact that certain geometric attributes of a knot imply membership in various levels of the filtration. We show the counterpart of this fact for two new filtrations of $\mathcal{C}$ due to Cochran--Harvey--Horn, the \textit{positive} and \textit{negative} filtrations, denoted by $\{\mathcal{P}_n\}_{n=0}^\infty$ and $\{\mathcal{N}_n\}_{n=0}^\infty$ respectively. In particular, we show that if a knot $K$ bounds a \textit{Casson tower} of height $n+2$ in $B^4$ with only positive (resp. negative) kinks in the base-level kinky disk, then $K\in\mathcal{P}_n$ (resp. $\mathcal{N}_n$). En route to this result we show that if a knot $K$ bounds a Casson tower of height $n+2$ in $B^4$, it bounds an embedded (symmetric) \textit{grope} of height $n+2$, and is therefore, $n$--solvable. We also define a variant of Casson towers and show that if $K$ bounds a tower of type $(2,\,n)$ in $B^4$, it is $n$--solvable. If $K$ bounds such a tower with only positive (resp. negative) kinks in the base-level kinky disk then $K\in \mathcal{P}_n$ (resp. $K\in \mathcal{N}_n$). Our results show that either every knot which bounds a Casson tower of height three is topologically slice or there exists a knot in $\bigcap \mathcal{F}_n$ which is not topologically slice. We also give a 3--dimensional characterization, up to concordance, of knots which bound kinky disks in $B^4$ with only positive (resp. negative) kinks; such knots form a subset of $\mathcal{P}_0$ (resp. $\mathcal{N}_0$).
\end{abstract}

\maketitle

\section{Introduction}
A \textit{knot} is the image of a smooth embedding $S^1 \hookrightarrow S^3 = \partial B^4$. A knot is called \textit{slice} if it bounds a smooth, properly embedded disk in $B^4$. The set of knots, modulo slice knots, under the connected sum operation forms an abelian group called the \textit{knot concordance group}, denoted by $\mathcal{C}$. We will often use the same letter to denote a knot $K$ and its concordance class. There is a parallel theory of concordance in the topological category. In particular, a knot is called \textit{topologically slice} if it bounds a proper, topologically embedded, locally flat disk in $B^4$. There exist infinitely many knots which are topologically slice but not smoothly slice (see, for example, \cite{En95,Gom86, HeK12, HeLivRu12, Hom11}). 

Much like the 3--dimensional study of knots frequently focuses on determining `how close a knot is to being unknotted', the 4--dimensional study attempts to assess `how close a knot is to being slice'. In 2003, this notion was formalized when Cochran--Orr--Teichner \cite{COT03} introduced the \textit{$n$--solvable filtration} of $\mathcal{C}$ and showed that the lower levels of the filtration encapsulate the information one can extract from various classical concordance invariants, such as algebraic concordance class, Levine--Tristram signatures, Casson--Gordon invariants, etc. Therefore, in an almost quantifiable sense, the deeper a knot is within the $n$--solvable filtration, the closer it is to being slice. Studying filtrations gives us a way of understanding the structure of $\mathcal{C}$, a large unwieldy object, in terms of smaller (and hopefully simpler) pieces. 

Part of the justification for the naturality of the $n$--solvable filtration is its close relationships with several more geometric filtrations of $\mathcal{C}$. In particular, certain geometric attributes imply membership in various levels of the $n$--solvable filtration, as follows. 

\begin{thm_1}[Theorems 8.11 and 8.12 of \cite{COT03}] If a knot $K$ bounds a grope of height $n+2$, then $K$ is $n$--solvable. If a knot $K$ bounds a Whitney tower of height $n+2$, then $K$ is $n$--solvable. \end{thm_1} 

Cochran--Harvey--Horn \cite{CHHo13} have recently introduced a new pair of filtrations (by monoids) of $\mathcal{C}$, the \textit{positive} and \textit{negative} filtrations: 
$$\cdots \subseteq \mathcal{P}_{n+1} \subseteq \mathcal{P}_n \subseteq \cdots \subseteq \mathcal{P}_0\subseteq \mathcal{C}$$
$$\cdots \subseteq \mathcal{N}_{n+1} \subseteq \mathcal{N}_n \subseteq \cdots \subseteq \mathcal{N}_0\subseteq \mathcal{C}$$
(see Section \ref{defns} for precise definitions). These new filtrations have proven to be of interest because they can be used to study smooth concordance classes of topologically slice knots; this distinguishes them from the $n$--solvable filtration, since if $K$ is topologically slice, $K$ is $n$--solvable for all $n$. Cochran--Harvey--Horn also defined the bipolar filtration (by subgroups) of $\mathcal{C}$, $\mathcal{B}_n:=\mathcal{P}_n\cap\,\mathcal{N}_n$ \cite{CHHo13}, and it is expected that this filtration will non-trivially filter topologically slice knots at each $n$, i.e.\ if $\mathcal{T}_n=\mathcal{B}_n \cap \{\text{topologically slice knots}\}$, it is expected that $\mathcal{T}_n\neq \mathcal{T}_{n+1}$. This is currently known for knots at $n\leq 1$ \cite{CHHo13, Cho12}. For links of two or more components, this is known for all $n$ by work of Cha--Powell \cite{ChaPow14}.

In this paper we will prove counterparts of Theorem 1 for the positive and negative filtrations in terms of \textit{Casson towers} \cite{Cas86, Freed82}: 4--dimensional objects built using layers of immersed disks (see Figure \ref{schematictower} for a schematic picture). In particular, we define several new filtrations of $\mathcal{C}$: $\{\mathfrak{C}_n\}_{n=1}^\infty$, $\{\mathfrak{C}_n^+\}_{n=1}^\infty$, $\{\mathfrak{C}_n^-\}_{n=1}^\infty$, $\{\mathfrak{C}_{2,\,n}\}_{n=1}^\infty$, $\{\mathfrak{C}_{2,\,n}^+\}_{n=1}^\infty$, and $\{\mathfrak{C}_{2,\,n}^-\}_{n=1}^\infty$.

\begin{figure}[ht!]
  \begin{center}
  \includegraphics[width=2in]{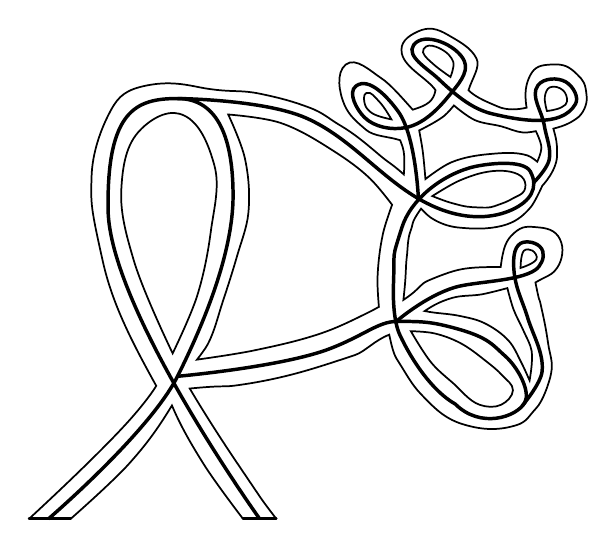}
  \caption{Schematic diagram of a Casson tower of height three.}\label{schematictower}
  \end{center}
\end{figure} 

Any knot $K$ that can be changed to a slice knot by only changing positive crossings to negative crossings is known to be in $\mathcal{P}_0$ by \cite[Proposition 3.1]{CHHo13} and \cite[Lemma 3.4]{CLic86}. Such a knot also bounds an immersed disk in $B^4$ with only positive self-intersections (i.e.\ \textit{kinks}). Indeed if a knot $K$ bounds an immersed disk in $B^4$ with only positive kinks, we can `blow up' the kinks, i.e.\, connect-sum with a $\mathbb{CP}(2)$ at each kink, to obtain a slice disk for $K$ in a 4--manifold with positive definite intersection form as called for in the definition for $\mathcal{P}_0$. (This reveals how the definition of $\mathcal{P}_0$ is a generalization of both the ordering on knot concordance classes given by \cite{CGom88} and \cite{CLic86}, and the notion of \textit{kinkiness} of knots defined by Gompf in \cite{Gom86}.) Similar statements hold for knots bounding immersed disks with only negative kinks and $\mathcal{N}_0$. Since bounding an immersed disk is closely  related to membership in the zero'th level of the positive and negative filtrations, Casson towers---built using layers of immersed disks---are natural objects to study in this context. 

In this paper, we establish several relationships between various filtrations of $\mathcal{C}$ (Theorem A) and completely characterize knots in $\mathfrak{C}_1^\pm$, i.e.\ knots which bound kinky disks in $B^4$ with only positive (resp. negative) kinks (Theorem B) as follows. 

\begin{thm_A}\label{inclusions} Let  $\{\mathcal{F}_n\}_{n=0}^\infty$ denote the $n$--solvable filtration of $\mathcal{C}$ and $\{\mathcal{G}_n\}_{n=0}^\infty$ the (symmetric) \textit{grope filtration} of $\mathcal{C}$. $\{\mathcal{G}_{2,\,n}\}_{n=0}^\infty$ is a slight enlargement of the grope filtration. (Precise definitions for the filtrations can be found in Section \ref{defns}.) 

For any $n\geq 0$,
\begin{itemize}
\item[(i)]$\mathfrak{C}_{n+2}\subseteq\mathcal{G}_{n+2}\subseteq\mathcal{F}_n$,
\item[(ii)]$\mathfrak{C}_{2,\,n}\subseteq\mathcal{G}_{2,\,n}\subseteq\mathcal{F}_n$,
\item[(iii)]$\mathfrak{C}_{n+2}^+ \subseteq \mathfrak{C}_{2,\,n}^+ \subseteq \mathcal{P}_n$,
\item[(iv)]$\mathfrak{C}_{n+2}^- \subseteq \mathfrak{C}_{2,\,n}^-\subseteq \mathcal{N}_n.$
\end{itemize}
\end{thm_A}

\begin{thm_B}For any knot $K$, the following statements are equivalent.
\begin{itemize}
\item[(a)] $K\in \mathfrak{C}_1^+$ (resp. $\mathfrak{C}_1^-$)
\item[(b)] $K$ is concordant to a fusion knot of split positive (resp. negative) Hopf links
\item[(c)] $K$ is concordant to a knot which can be changed to a ribbon knot by changing only positive (resp. negative) crossings.
\end{itemize}
\end{thm_B}
The second inclusion in part (ii) of Theorem A is exactly the second result listed earlier in Theorem 1 \cite[Theorem 8.11]{COT03} and we only include it here for completeness. 

Let $\mathcal{W}_n$ denote the set of knots which bound \textit{Whitney towers} of height $n$ in $B^4$. Whitney towers are similar to Casson towers except that kinks appear in pairs of opposing sign and higher-stage disks are attached to curves which traverse from one kink in a pair to the other (see \cite{FreedQ90} for more details).  It is well-known that any Casson tower yields a Whitney tower with the same attaching curve; one may see this using Kirby diagrams (the basic idea is that, in a Casson tower, we can locally introduce a kink of the opposite sign at any kink in such a way that the attaching curve for the higher-stage disk in the Casson tower is changed appropriately). As a result, in conjunction with Theorem 1 \cite[Theorem 8.12]{COT03}, it was already known that $\mathfrak{C}_{n+2}\subseteq\mathcal{W}_{n+2}\subseteq\mathcal{F}_n$. Our contribution consists of showing that if a knot bounds a Casson tower $T$ of height $n$ in $B^4$, it bounds a properly embedded grope of height $n$ within $T$ (Proposition \ref{towersandgropes}). In contrast, Schneiderman has shown that if a knot bounds a properly embedded grope of height $n$ in $B^4$, it bounds a Whitney tower of height $n$ in $B^4$ \cite[Corollary 2]{Schn06}. The converse to Schneiderman's statement is not known. In summary, it was previously known that $\mathcal{G}_{n+2}\subseteq\mathcal{W}_{n+2}\subseteq\mathcal{F}_n$ and $\mathfrak{C}_{n+2}\subseteq\mathcal{W}_{n+2}\subseteq\mathcal{F}_n$. We have now shown that $\mathfrak{C}_{n+2}\subseteq\mathcal{G}_{n+2}\subseteq\mathcal{W}_{n+2}\subseteq\mathcal{F}_n$.

We will see that $\mathfrak{C}_n^\pm\subseteq\mathfrak{C}_n$ and $\mathfrak{C}_{2,\,n}^\pm\subseteq\mathfrak{C}_{2,\,n}$ for all $n$, and therefore parts (i) and (ii) of Theorem A imply that $\mathfrak{C}_{n+2}^\pm\subseteq \mathcal{F}_n$ and $\mathfrak{C}_{2,\,n}^\pm\subseteq\mathcal{F}_n$. Along with \cite[Proposition 5.5]{CHHo13} which states that $\mathcal{P}_n\subseteq\mathcal{F}_n^\text{odd}$ (and $\mathcal{N}_n\subseteq\mathcal{F}_n^\text{odd}$), we get the following inclusions for each $n$.  ($\{\mathcal{F}_n^\text{odd}\}_{n=0}^\infty$ is a larger filtration than the $n$--solvable filtration, i.e.\ $\mathcal{F}_n\subseteq\mathcal{F}_n^\text{odd}$ for each $n$.)

\vspace{10pt}
\hspace*{\fill}
$\begin{array}[c]{ccc}
\mathcal{F}_n&\subseteq &\mathcal{F}_n^\text{odd}\\
\rotatebox{90}{$\subseteq$}&&\rotatebox{90}{$\subseteq$}\\
\mathfrak{C}_{n+2}^+&\subseteq&\mathcal{P}_n
\end{array}$
\hspace*{\fill}
$\begin{array}[c]{ccc}
\mathcal{F}_n&\subseteq &\mathcal{F}_n^\text{odd}\\
\rotatebox{90}{$\subseteq$}&&\rotatebox{90}{$\subseteq$}\\
\mathfrak{C}_{2,\,n}^+&\subseteq&\mathcal{P}_n
\end{array}$
\hspace*{\fill}
$\begin{array}[c]{ccc}
\mathcal{F}_n&\subseteq &\mathcal{F}_n^\text{odd}\\
\rotatebox{90}{$\subseteq$}&&\rotatebox{90}{$\subseteq$}\\
\mathfrak{C}_{n+2}^-&\subseteq&\mathcal{N}_n
\end{array}$
\hspace*{\fill}
$\begin{array}[c]{ccc}
\mathcal{F}_n&\subseteq &\mathcal{F}_n^\text{odd}\\
\rotatebox{90}{$\subseteq$}&&\rotatebox{90}{$\subseteq$}\\
\mathfrak{C}_{2,\,n}^-&\subseteq&\mathcal{N}_n
\end{array}$
\hspace*{\fill}
\vspace{10pt}

We state the following corollaries to facilitate easy reference in our proofs and examples. They may be considered to be corollaries of Theorem A or of Theorem 1 along with the fact that Casson towers yield Whitney towers with the same attaching curve.

\begin{cor_1}If a knot $K$ lies in $\mathfrak{C}_2$, $\Arf\,(K)=0$.\end{cor_1}

\begin{cor_2}If a knot $K$ lies in $\mathfrak{C}_{2,\,1}$, then $K$ is algebraically slice.\end{cor_2}

The above statements follow easily from well-known properties of the $n$--solvable filtration, namely, any knot in $\mathcal{F}_0$ has trivial Arf invariant and any knot in $\mathcal{F}_1$ is algebraically slice \cite{COT03}.

Gompf--Singh's refinement of Freedman's Reimbedding Theorem for Casson towers \cite[Theorem 4.4]{Freed82}\cite[Theorem 5.1]{GomSin84} implies that the filtrations $\{\mathfrak{C}_n\}_{n=1}^\infty$ and $\{\mathfrak{C}_n^\pm\}_{n=1}^\infty$ stablize at $n=5$, i.e.\ $\mathfrak{C}_5 = \mathfrak{C}_6 = \mathfrak{C}_7 = \cdots$ and $\mathfrak{C}_5^\pm = \mathfrak{C}_6^\pm = \mathfrak{C}_7^\pm = \cdots$. In fact, as we describe in Section \ref{defns}, $\mathfrak{C}_5$ is \textit{equal} to $\mathcal{T}$, the set of all topologically slice knots. It is possible (and is believed by some experts) that $\mathfrak{C}_3$ is equal to $\mathcal{T}$. It is worth noting that while $\mathfrak{C}_5=\mathcal{T}$, each of $\mathfrak{C}_5^\pm$ is a \textit{proper} subset of $\mathcal{T}$. This mirrors the fact that the positive/negative filtrations are able to distinguish topologically slice knots while the $n$--solvable filtration cannot. 

As we see above the $\{\mathfrak{C}_{n}\}_{n=1}^\infty$ filtration stabilizes at $n=5$ (or conjecturally at $n=3$). It is also easy to see that $\mathfrak{C}_1=\mathcal{C}$, i.e.\ any knot bounds an immersed disk in $B^4$. This indicates that if one is interested in studying smooth concordance classes of topologically slice knots one should focus on these levels. The filtration $\{\mathfrak{C}_{2,\,n}\}_{n=0}^\infty$ is designed specifically to filter knots within these levels, in particular, between $\mathfrak{C}_2$ and $\mathfrak{C}_3$. 

We also see, in Corollary \ref{height3}, that $\mathfrak{C}_3\subseteq \mathfrak{C}_{2,\,n}$ and $\mathfrak{C}_3^\pm\subseteq \mathfrak{C}_{2,\,n}^\pm$ for all $n\geq 0$. From part (i) of Theorem A then,
$$\mathfrak{C}_3\subseteq\bigcap_{n=0}^\infty\mathcal{F}_n.$$ 
The only presently known elements of $\bigcap_{n=0}^\infty\mathcal{F}_n$ are topologically slice knots and it is conjectured that $\bigcap_{n=0}^\infty\mathcal{F}_n=\mathcal{T}$. From the above, we can infer that either any knot bounding a Casson tower of height three is topologically slice or there exist knots in $\bigcap_{n=0}^\infty\mathcal{F}_n$ which are not topologically slice. Similarly, since $$\bigcap_{n=0}^\infty\mathfrak{C}_{2,\,n}\subseteq\bigcap_{n=0}^\infty\mathcal{F}_n$$ we are led to conjecture that any knot in $\bigcap \mathfrak{C}_{2,\,n}$ is topologically slice. 

By parts (iii) and (iv) of Theorem A,
$$\mathfrak{C}_3^+\subseteq\bigcap_{n=0}^\infty\mathcal{P}_n \text{ and } \mathfrak{C}_3^-\subseteq \bigcap_{n=0}^\infty\mathcal{N}_n.$$

This indicates that membership in $\mathfrak{C}_3^\pm$ is a very restrictive condition. For example, the results of \cite{CHHo13} show how membership in just the zero'th and first levels of the positive and negative filtrations impose severe restrictions on smooth concordance class. This also reveals that while the positive and negative filtrations have had success in distinguishing concordance classes of topologically slice knots, they cannot be used to distinguish between topologically slice knots in $\mathfrak{C}^\pm_3$. 

\subsection{Organization of the paper}

We will start by stating precise definitions of Casson towers and the various filtrations of $\mathcal{C}$ in Section \ref{defns}. Sections \ref{tower_results} and \ref{c1+} consist of the proofs of Theorems A and B respectively; additionally in Section \ref{c1+} we give an overview of various notions of positivity of knots and how membership in $\mathcal{P}_0$ and $\mathfrak{C}_1^+$ are related to them. In Section \ref{exprop} we will list various properties of the Casson tower filtrations. We generalize our results to the case of (string) links in Section \ref{links}.

\subsection{Acknowledgements}

This paper was completed as part of the author's doctoral work at Rice University. The author would like to thank her advisor Tim Cochran for his time, guidance, and continued encouragement. Special thanks and much gratitude are also due to Robert Gompf for his insights into Casson towers and his patience in replying to the author's persistent emails, to Jae Choon Cha for his helpful comments on an earlier draft, and to the anonymous referee whose detailed and thoughtful remarks and suggestions significantly improved this paper. The author was partially supported by NSF--DMS--1309081 and the Nettie S.\ Autrey Fellowship (Rice University). 

\section{Notation and definitions}\label{defns}

\subsection{Casson towers}Suppose $f: D \rightarrow M$ is a smooth self-transverse immersion, where $D$ is a genus zero, oriented 2--manifold, $M$ is an oriented, smooth 4--manifold, and $f^{-1}(\partial M)=\partial D$. We will refer to the points of self-intersection of $f(D)$ as \textit{kinks} and $f(D)$ as being \textit{kinky}. In this paper we will only use kinky disks, that is, the case where $D$ is a disk. Since $M$ and $D$ are oriented, each kink of $f(D)$ has a canonical sign. A regular neighborhood of a kinky disk in a 4--manifold will be called a \textit{kinky handle}. For a kinky handle which is a regular neighborhood of the kinky disk $f(D)$, the \textit{attaching curve} is the simple closed curve $f(\partial D)$. 

\begin{figure}[t!]
  \begin{center}
  \includegraphics[width=3in]{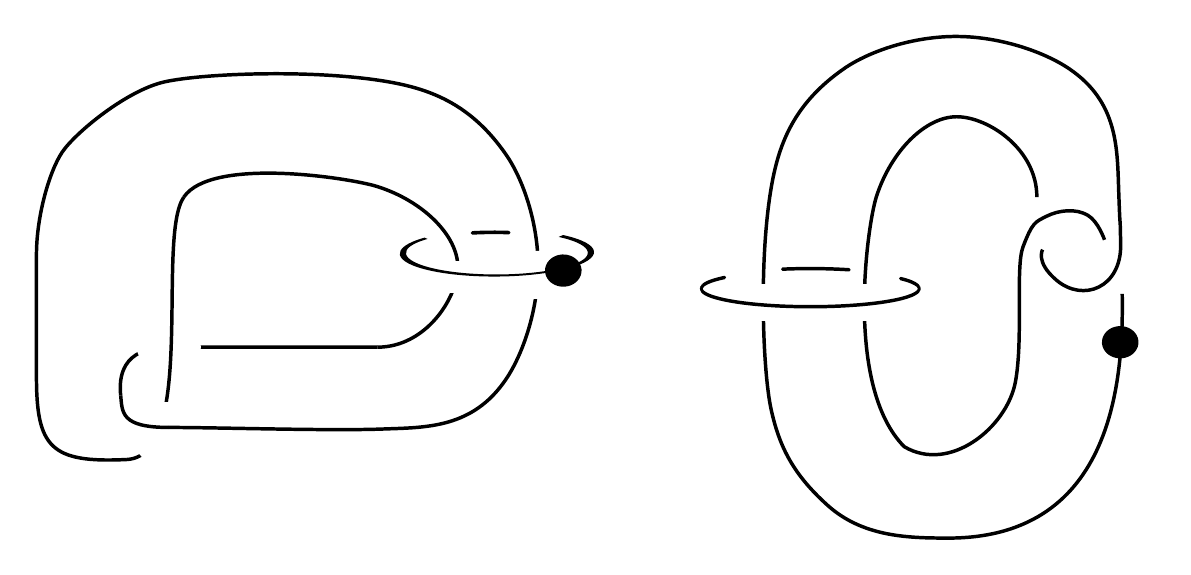}
  \caption{Two Kirby diagrams for a kinky handle with a single positive kink. The two panels are pictures of the same space and differ only by an isotopy of curves; we show both versions since each will appear later in the paper. The dotted curve represents a 1--handle, and the other curve is the attaching curve for the kinky handle.}\label{kinkyex}
  \end{center}
\end{figure}

In our proofs we will frequently utilize Kirby diagrams to describe 4--manifolds. Background on Kirby diagrams and Kirby calculus can be found in \cite{GomStip99}. Kirby diagrams for a kinky handle with a single positive kink are given in Figure \ref{kinkyex}, where the sign of the clasp corresponds to the sign of the kink. To obtain pictures for a kinky handle with a single negative kink, we need simply to use the negative clasp. It is important to note that the leftmost (undecorated) curves in the two diagrams do not represent attaching circles for handles but rather the attaching curve for the kinky handle itself. This indicates that a kinky handle with a single kink is diffeomorphic to $S^1\times D^3$ since it has a Kirby diagram consisting of a single dotted circle. A Kirby diagram for a kinky handle with $n$ kinks would consist of $n$ unlinked unknotted circles decorated with dots (indicating correctly that the corresponding 4--manifold is diffeomorphic to $\natural_n S^1\times D^3$) with an (undecorated) attaching curve passing through each dotted circle to clasp itself according to the sign of the kink (see Figure \ref{height2tower}). For more details, the interested reader is directed to Chapter 6 of \cite{GomStip99}. 

The 0--framed meridians of the dotted circles in the Kirby picture for a kinky handle $\kappa$ form the \textit{standard set of curves} for $\kappa$; this set is characterized by the property that if we were to attach 2--handles to these (framed) curves, the resulting 4--manifold $(\kappa, \text{attaching curve})$ would be diffeomorphic to the standard 2--handle $(D^2\times D^2, \partial D^2\times \{0\})$. There is also a notion of a canonical framing for the attaching curve of a kinky handle $\kappa$, namely the unique framing such that if a 2--handle were attaching to $\kappa$ along the attaching curve with that framing, the resulting 4--manifold would have intersection form zero. This is equivalent to saying that  if one pushes off a parallel copy of the attaching curve by its canonical framing, the two circles should bound disjoint embedded surfaces inside $\kappa$. Note that this notion is distinct from the framing one gets from the normal bundle of the core kinky disk for $\kappa$, and in fact these two notions differ by exactly twice the number of (signed) self-intersections of the core kinky disk for $\kappa$ (see \cite{GomSin84}).

Using kinky handles we may construct a \textit{Casson tower}. Detailed descriptions of Casson towers may be found in \cite{Cas86, Freed82, GomSin84}. A Casson tower of height one is simply a kinky handle. A Casson tower of height two is obtained from a Casson tower of height one, $T_1$, by attaching kinky handles to each member of a standard set of curves for $T_1$ by matching the framings (recall that for any kinky handle the attaching curve and each member of the standard set of curves are framed). We refer to these newly attached kinky handles as the `second-stage kinky handles'. Suppose that towers of height $n$ have been defined. We define a \textit{standard set of curves} for a height $n$ tower to be the union of standard sets of curves for each $n^\text{th}$--stage kinky handle. We then construct a height $n+1$ Casson tower by attaching kinky handles (`$n+1^\text{th}$--stage kinky handles') to each member of a standard set of curves for a Casson tower of height $n$. The corresponding infinite construction, i.e.\ a Casson tower with infinite height, is called a \textit{Casson handle}. 

\begin{figure}[t!]
  \begin{center}
  \includegraphics[width=\textwidth]{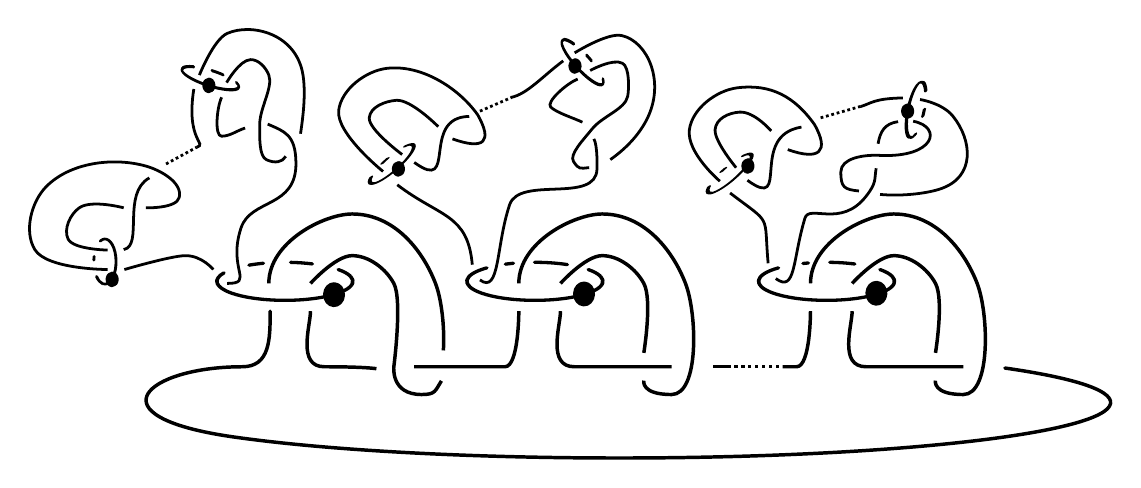}
  \put(-5,0.9){0}
  \put(-3.25,1.3){0}
  \put(-1.82,1.1){0}
  \caption{A Kirby diagram for a general Casson tower of height two. The bottommost curve in the picture is the attaching curve.}\label{height2tower}
  \end{center}
\end{figure}

We will consider every Casson tower to have a fixed decomposition into kinky handles. A Kirby diagram for a general Casson tower of height two is shown in Figure \ref{height2tower}. The \textit{attaching curve} for a Casson tower or handle is the attaching curve for the first-stage kinky handle; in Figure \ref{height2tower} it appears as the bottommost (undecorated) curve. In the Kirby diagram, the parallel of the attaching curve with linking number zero is the push off along the canonical framing (we can infer this from the fact that the two curves bound disjoint surfaces in the first-stage kinky handle, as we see at the beginning of the proof of Proposition \ref{towersandgropes}). The standard set of curves for a Casson tower appear in the diagram as (0--framed) meridians of the dotted circles of the last layer of kinky handles, that is, simple loops traversing the terminal 1--handles exactly once; note that these curves generate the fundamental group of the Casson tower (since a Casson tower is diffeomorphic to $\natural S^1\times D^3$, its fundamental group is free). Sometimes we will also refer to the meridians of the dotted circles at a given stage within a Casson tower. For example, we might refer to the standard set of curves at the second stage of a Casson tower of height four. If a stage of the Casson tower is not specified, we refer to the standard set of curves at the terminal stage. 

Every Casson tower has a 2--complex as a strong deformation retract, called its \textit{core}. For a Casson tower of height one, namely the regular neighborhood of a kinky disk $D$, the core is exactly $D$. For a Casson tower of greater height, the core consists of the cores of each kinky handle along with certain canonical annuli. This is described in greater detail in \cite[Section 2.2.6]{GomSin84}.

We will say that a curve $\gamma\subseteq \partial M$ which is null-homologous in $\partial M$ `bounds a Casson tower $T$ in a 4--manifold $M$' if there is a proper embedding of $T$ in $M$ where a 0--framed regular neighborhood of the attaching curve of $T$ (seen in a Kirby diagram for $T$) is identified with a 0--framed neighborhood of $\gamma$ in $\partial M$. If the 4--manifold is not mentioned, the reader should assume it to be $B^4$. In particular, this means that if a knot $K$ is said to bound, say, the Casson tower $T$ shown in Figure \ref{height2tower}, the 0--framed longitude of $K$ in $S^3$ can be seen as the 0--framed longitude of the attaching curve of $T$, i.e.\ it appears as the parallel of the attaching curve with zero linking number.

Recall that for any group $G$, $G^{(n)}$ denotes the $n^{\text{th}}$ term of its derived series. 

\begin{defn_1}A knot $K$ is said to be in $\mathfrak{C}_n$ if it bounds a Casson tower of height $n$. \end{defn_1}

\begin{defn_2}A knot $K$ is said to be in $\mathfrak{C}_{2,\,n}$ if it bounds a Casson tower $T$ of height two such that each member of a standard set of curves for $T$ is in $\pi_1(B^4 -  C)^{(n)}$, where $C$ is the core of $T$. \end{defn_2}

Each $\mathfrak{C}_n$ and $\mathfrak{C}_{2,\,n}$ is a subgroup of $\mathcal{C}$ with respect to the connected sum operation on knot concordance classes. 

\begin{defn_3}A knot $K$ is said to be in $\mathfrak{C}_n^+$ (resp. $\mathfrak{C}_n^-$) if it bounds a Casson tower of height $n$ such that the base-level kinks are all positive (resp. negative). \end{defn_3}

\begin{defn_4}A knot $K$ is said to be in $\mathfrak{C}_{2,\,n}^+$ (resp. $\mathfrak{C}_{2,\,n}^-$) if it bounds a Casson tower $T$ of height two such that the base-level kinks are all positive (resp. negative) and each member of a standard set of curves for $T$ is in $\pi_1(B^4 -  C)^{(n)}$, where $C$ is the core of $T$. \end{defn_4}

Each $\mathfrak{C}_n^\pm$  and $\mathfrak{C}_{2,\,n}^\pm$ is a monoid with respect to the connected sum operation on knot concordance classes. They are \textit{not} subgroups of $\mathcal{C}$, since if $K\in\mathfrak{C}_n^+$, $-K\in\mathfrak{C}_n^-$ but $-K$ may not be in $\mathfrak{C}_n^+$; and if $K\in\mathfrak{C}_{2,\,n}^+$, $-K\in\mathfrak{C}_{2,\,n}^-$ but $-K$ may not be in $\mathfrak{C}_{2,\,n}^+$. 

We will sometimes use the notation $\mathfrak{C}_n^{\pm}$ when referring to either of $\mathfrak{C}_n^+$ or $\mathfrak{C}_n^-$. Clearly, 
$$\cdots \mathfrak{C}_{n+1}^\pm \subseteq \mathfrak{C}_n^\pm \subseteq \cdots \subseteq \mathfrak{C}_1^\pm \subseteq \mathcal{C}$$
$$\cdots \mathfrak{C}_{n+1} \subseteq \mathfrak{C}_n \subseteq \cdots \subseteq \mathfrak{C}_1 \subseteq \mathcal{C}$$
and
$$\cdots \mathfrak{C}_{2,\,n+1}^\pm \subseteq \mathfrak{C}_{2,\,n}^\pm \subseteq \cdots \subseteq \mathfrak{C}_{2,\,1}^\pm \subseteq \mathfrak{C}_{2,\,0}^\pm \equiv \mathfrak{C}_2^\pm \subseteq \mathcal{C}$$
$$\cdots \mathfrak{C}_{2,\,n+1} \subseteq \mathfrak{C}_{2,\,n} \subseteq \cdots \subseteq \mathfrak{C}_{2,\,1} \subseteq \mathfrak{C}_{2,\,0} \equiv \mathfrak{C}_2 \subseteq \mathcal{C}$$

Studying the filtrations $\{\mathfrak{C}_{n}\}_{n=1}^\infty$ is unsatisfying in general since $\mathfrak{C}_5 = \mathfrak{C}_6 = \mathfrak{C}_7 = \cdots$. As we mentioned in the introduction, this is due to Freedman's Reimbedding Theorem \cite[Theorem 4.4]{Freed82} (later improved by Gompf--Singh in \cite[Theorem 5.1]{GomSin84}) which states that any Casson tower of height five contains within it arbitrarily high Casson towers sharing its initial three stages. In particular, this allows us to see that a Casson tower of height five contains a Casson handle within it. Along with Freedman's extraordinary theorem that any Casson handle is homeomorphic to an open 2--handle \cite[Theorem 1.1]{Freed82}, this implies that if a knot bounds a Casson tower $T$ of height five, it has a topological slice disk within $T$ itself. 

The question of whether a given Casson tower contains a topological slice disk for its attaching curve can be rephrased in terms of whether a certain iterated, ramified Whitehead double of the Hopf link is topologically slice in the 4--ball where all but one of the slice disks is standard. (This relationship can be easily seen using Kirby diagrams and is indicated in \cite[pp. 80--81]{Kir89}.) Using this connection it is easy to infer that not all Casson towers of height one or two contain topological disks. The simplest Casson towers of height three and four (i.e.\ with a single kink at each stage) contain topological slice disks for the attaching curve \cite{Freed88}, but this is not known for such towers in general. It appears to be widely believed by experts that all Casson towers of height three and higher contain topological slice disks for the attaching curve\footnote{The current literature is somewhat misleading on the status of this conjecture for general Casson towers of height three and four.}.

Let $\mathcal{T}$ denote the set of all topologically slice knots. The above shows that if a knot bounds a `tall enough' Casson tower (height five is sufficient, height three is conjectured to be enough), it is topologically slice. That is, $\mathfrak{C}_5\subseteq \mathcal{T}$. Indeed, a result of Quinn \cite[Proposition 2.2.4]{Q82}\cite[Theorem 5.2]{Gom05} shows that any topologically slice knot bounds a Casson handle in $B^4$. Therefore, $\mathfrak{C}_5$ is \textit{equal} to $\mathcal{T}$. If every Casson tower of height three contains a topological slice disk for its attaching curve, $\mathfrak{C}_3$ would be equal to $\mathcal{T}$. 

\subsection{Filtrations of the knot concordance group}We end this section by recalling the definitions of several filtrations of $\mathcal{C}$.

\begin{definition}[Definition 2.2 of \cite{CHHo13}]For any $n\geq 0$, a knot $K\subseteq S^3$ is in $\mathcal{P}_n$ (resp. $\mathcal{N}_n$) and is said to be $n$--positive (resp. $n$--negative) if there exists a smooth, compact, oriented 4--manifold $V$ such that there is a properly embedded, smooth 2--disk $\Delta\subseteq V$ with $\partial\Delta=K$, $\partial V= S^3$, $[\Delta]$ trivial in $H_2(V, S^3)$ and 
\begin{enumerate}
\item $\pi_1(V)=0$
\item the intersection form on $H_2(V)$ is positive definite (resp. negative definite)
\item $H_2(V)$ has a basis represented by a collection of surfaces $\{S_i\}$ disjointly embedded in the exterior of $\Delta$ such that $\pi_1(S_i)\subseteq \pi_1(V -  \Delta)^{(n)}$ for all $i$.
\end{enumerate}
\end{definition}

\begin{definition}[\cite{COT03}]For any $n\geq 0$, a knot $K\subseteq  S^3$ is in $\mathcal{F}_n$ and is said to be $n$--solvable if there exists a smooth, compact, oriented 4--manifold $V$ such that there is a properly embedded, smooth 2--disk $\Delta\subseteq V$ with $\partial\Delta=K$, $\partial V= S^3$, $[\Delta]$ trivial in $H_2(V, S^3)$ and 
\begin{enumerate}
\item $H_1(V)=0$
\item there exist surfaces $\{L_1, D_1, L_2, D_2, \cdots, L_k,D_k\}$ (with product neighborhoods) embedded in $V -  \Delta$ which form an ordered basis for $H_2(V)$ such that 
\begin{enumerate}
\item for each $i$, $L_i$ and $D_i$ intersect transversely and positively exactly once
\item $L_i\cap D_j$, $L_i\cap L_j$, and $D_i\cap D_j$ are each empty if $i\neq j$
\item $\pi_1(L_i)\subseteq \pi_1(V -  \Delta)^{(n)}$ for all $i$
\item $\pi_1(D_i)\subseteq \pi_1(V -  \Delta)^{(n)}$ for all $i$.
\end{enumerate}
\end{enumerate}
\end{definition}

\begin{remark} The above definition appears different from the original definition of $n$--solvability in \cite{COT03} at first glance, but the equivalence between the two definitions is straightforward and we refrain from including the proof here. (A proof for the equivalence between the corresponding definitions for the $n$--positive filtration can be found in \cite[Proposition 5.2]{CHHo13}.) 

The original definition of the $n$--solvable filtration in \cite{COT03} was concerned with the \textit{topological} knot concordance group. Here, as in several recent works in the literature, we are using a version of the filtration for the \textit{smooth} knot concordance group. \end{remark}

If the $D_i$ in the above definition are not required to have product neighborhoods, we get a slight enlargement of the $n$--solvable filtration, $\{\mathcal{F}_n^\text{odd}\}_{n=0}^\infty$.

\begin{definition}A \textit{grope} is a pair (2--complex, attaching circle). A grope of height one is a compact, oriented surface $\Sigma$ with a single boundary component, \textit{the attaching circle}. Gropes of greater height are defined recursively as follows. Let $\{\alpha_i,\beta_i : i=1, \cdots, g\}$ be disjointly embedded curves representing a symplectic basis for $H_1(\Sigma)$, where $\Sigma$ is a grope of height one. A grope of height $n$ is obtained by attaching gropes of height $n-1$ along its attaching circle to each $\alpha_i$ and $\beta_i$ in $\Sigma$. \end{definition}

\begin{remark}The above gropes are sometimes referred to as `symmetric' gropes and therefore, the following construction is sometimes referred to as the \textit{symmetric grope filtration}.\end{remark}

\begin{definition}[\cite{COT03}]For any $n\geq 1$, a knot $K\subseteq S^3$ is in $\mathcal{G}_n$ if $K$ extends to a proper embedding of a grope of height $n$ with its untwisted framing in $B^4$. This gives the \textit{grope filtration of} $\mathcal{C}$, $\{\mathcal{G}_n\}_{n=1}^\infty$.\end{definition}

\begin{definition}For any $n\geq 0$, a knot $K\subseteq S^3$ is in $\mathcal{G}_{2,\,n}$ if $K$ extends to a proper embedding of a grope $G$ of height two with its untwisted framing in $B^4$ such that pushoffs of each member of a symplectic basis for the first homology groups of the second stage surfaces of $G$ are in $\pi_1(B^4 -  G)^{(n)}$.\end{definition}

\begin{remark}The groups $\mathcal{G}_{2,\,n}$ defined above have not appeared in the literature before to the author's knowledge. However, several proofs of results related to the grope filtration hold for the filtration $\{\mathcal{G}_{2,\,n}\}_{n=0}^\infty$; this is perhaps unsurprising since it is easily seen that $\mathcal{G}_{n+2}\subseteq\mathcal{G}_{2,\,n}$ for each $n$. The following is an example of such a result. \end{remark}

\begin{theorem}[Theorem 8.11 of \cite{COT03}]\label{g2n}$\mathcal{G}_{n+2}\subseteq\mathcal{G}_{2,\,n}\subseteq\mathcal{F}_n$ for each $n$.\end{theorem}

\begin{proof}Suppose a knot $K$ bounds a grope $G$ in $B^4$. If a curve on the second stage surfaces bounds a grope of height $n$ away from the first first two stages (call it $G'$), the curve lies in $\pi_1(B^4-G')^{(n)}$; as a result the first inclusion is clear. 

The second inclusion follows very easily from a close reading of the proof of \cite[Theorem 8.11]{COT03} (Theorem 1) where they show that $\mathcal{G}_{n+2}\subseteq \mathcal{F}_n$. Briefly, given a grope $G$ of height $n+2$ bounded by a knot, they only use the first two stages (call it $G'$) and the fact that a symplectic basis for $H_1$ of each second stage surface is in $\pi_1(B^4- G')^{(n)}$.\end{proof}

\section{Casson towers and various filtrations of the smooth knot concordance group}\label{tower_results}

In this section we prove several results connecting the types of Casson towers bounded by a knot $K$ and membership within the many filtrations of $\mathcal{C}$. Together these results comprise Theorem A. 
 
\begin{proposition}\label{towersandgropes}The attaching curve of a Casson tower $T$ of height $n$ bounds a properly embedded grope of height $n$ within $T$.\end{proposition}

\begin{figure}[t!]
  \begin{center}
  \includegraphics[width=5in]{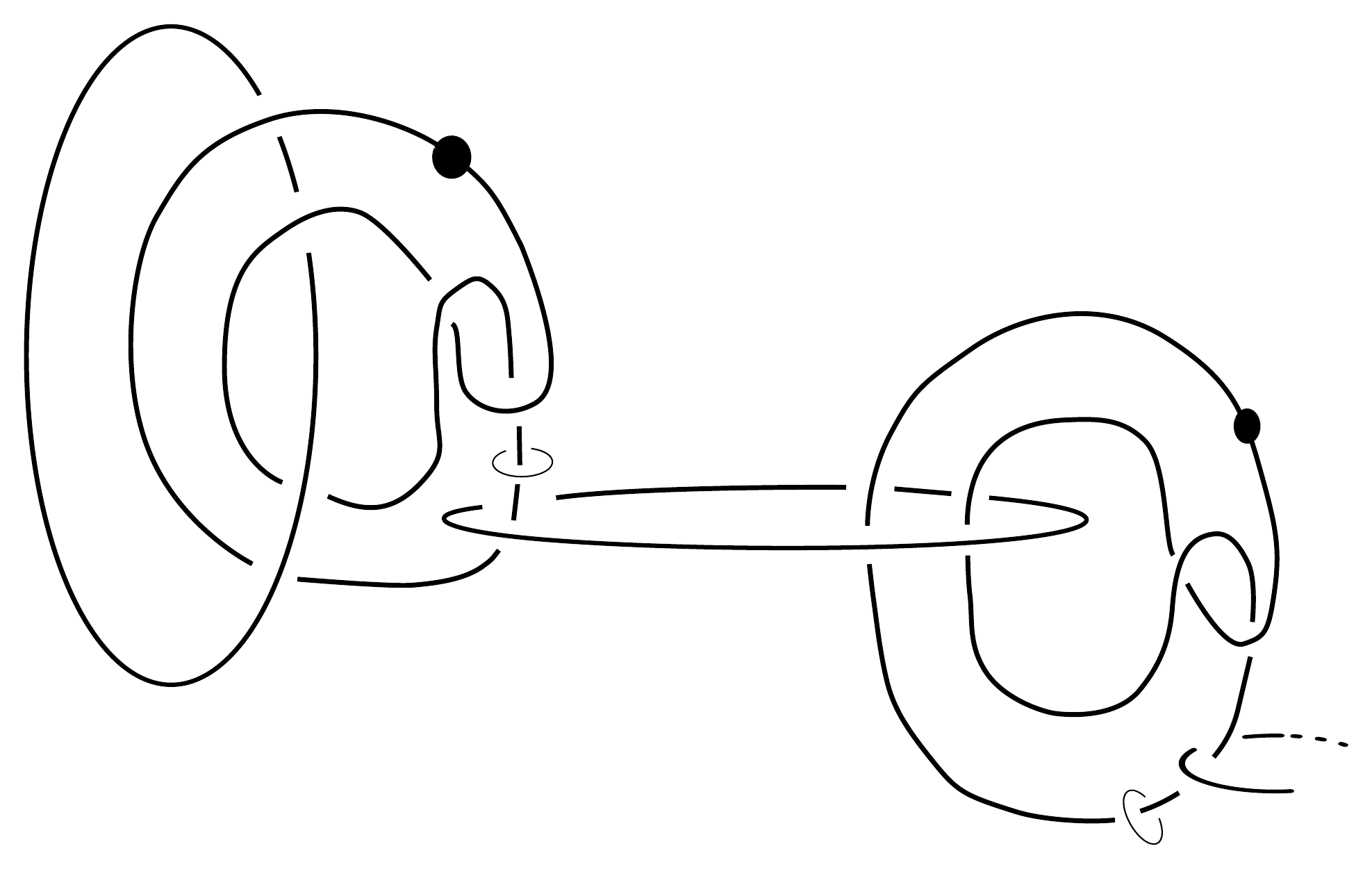}
  \put(-2.5,1.45){0}
  \put(-0.4,0.05){0}   
  \put(-2.95,1.45){$\alpha_1$}
  \put(-0.75,-0.05){$\alpha_2$}
  \caption{Proof of Proposition \ref{towersandgropes}: A Kirby diagram for the first two stages of a Casson tower with a single kink at each stage.}\label{towersandgropesfig1}
  \end{center}
\end{figure}

\begin{figure}[t!]
  \centering
  \includegraphics[width=5in]{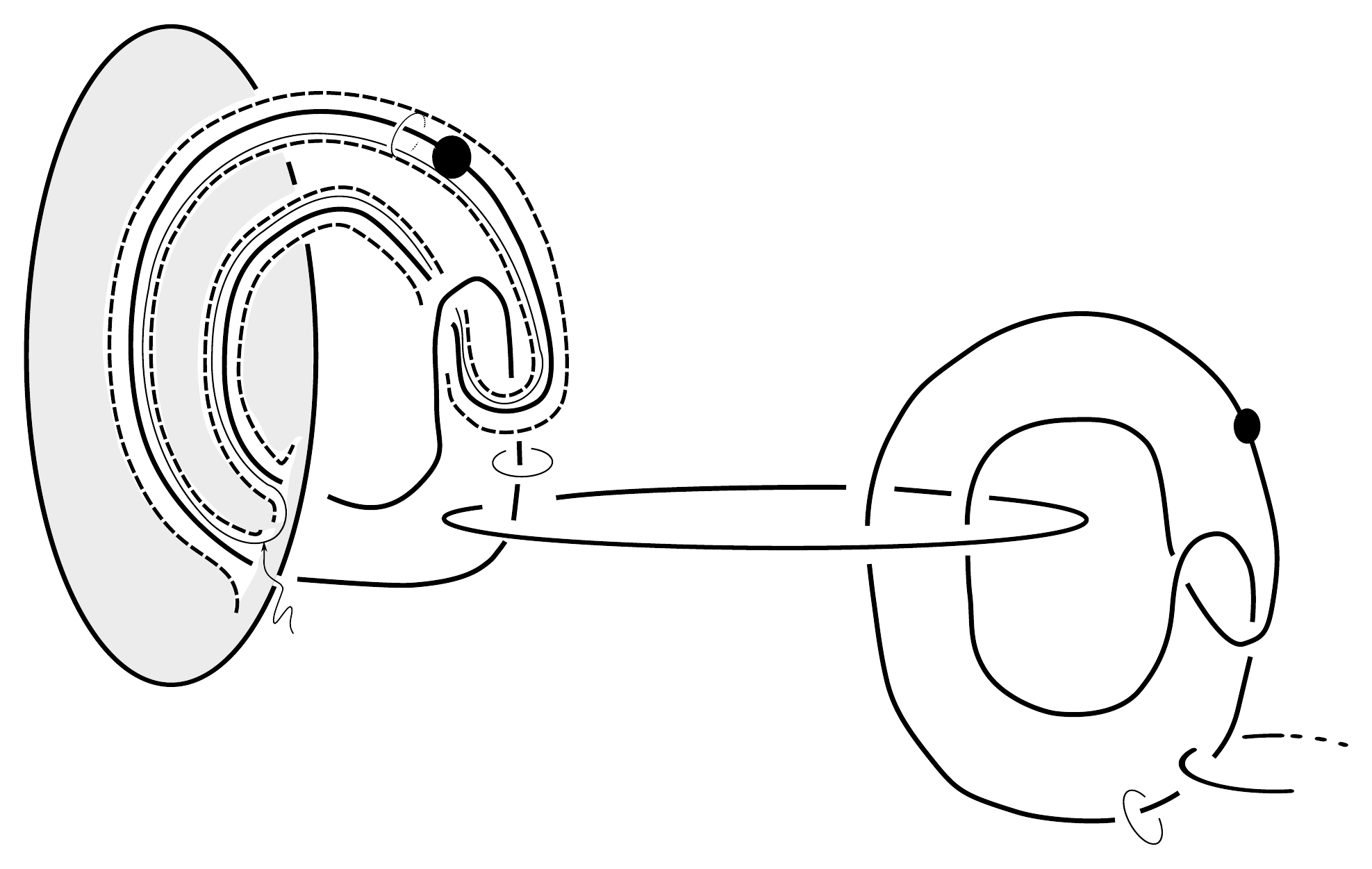}
  \put(-2.5,1.45){0}
  \put(-0.4,0.05){0}   
  \put(-2.95,1.45){$\alpha_1$}
  \put(-0.75,-0.05){$\alpha_2$}
  \put(-3.5,2.8){$m$}
  \put(-3.93,0.75){$\ell$}  
  \caption{Proof of Proposition \ref{towersandgropes}, Step 1: $\Sigma$, the first stage of the grope, consists of the standard disk bounded by the attaching curve with a tube (dashed) along the dotted circle. $m$ and $\ell$ denote the meridian and longitude respectively.}\label{towersandgropesfig2}
\end{figure}

\begin{figure}[ht!]
  \begin{center}
  \includegraphics[width=5in]{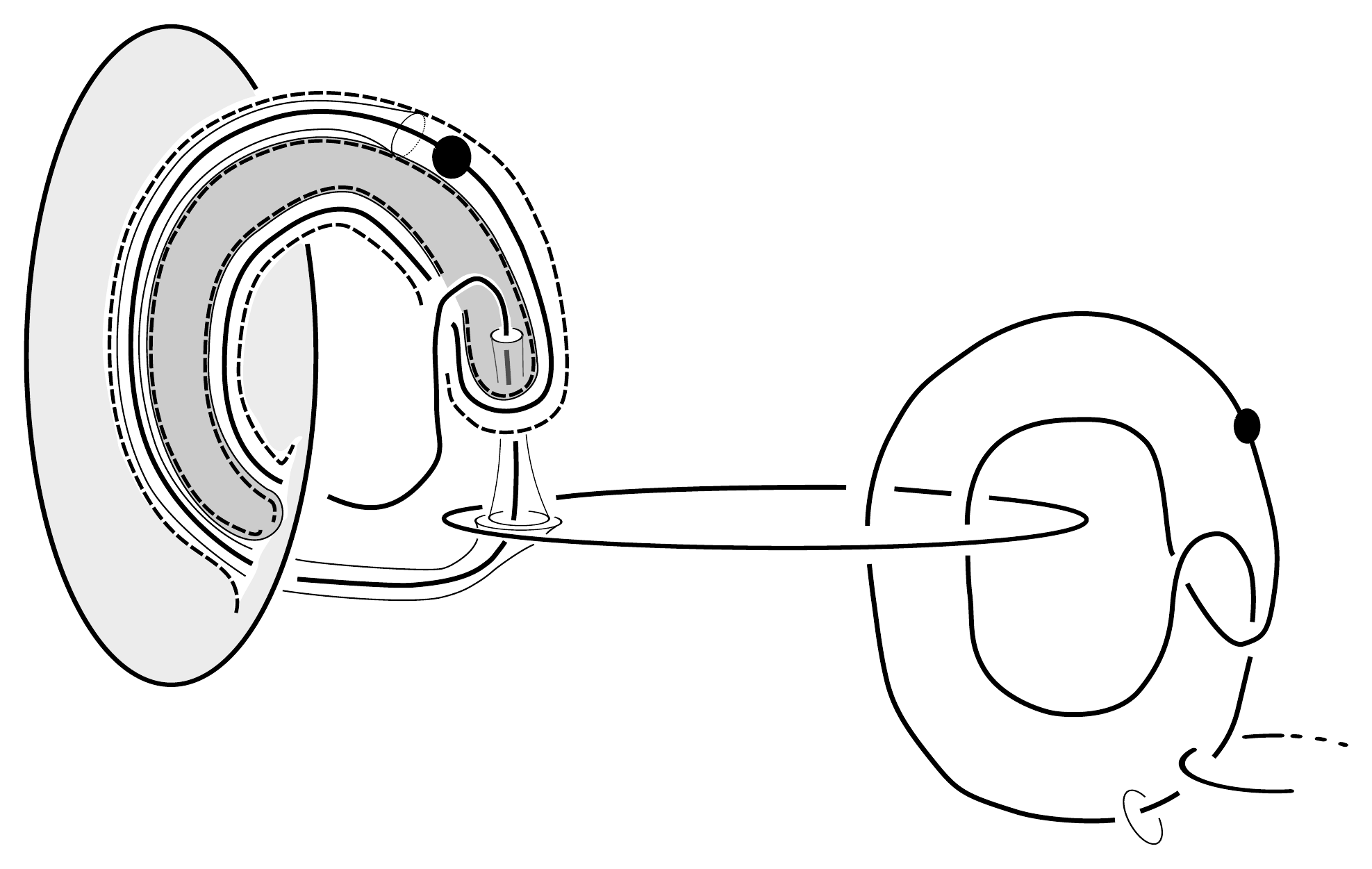}
  \put(-2.5,1.45){0}
  \put(-0.4,0.05){0}   
  \put(-0.75,-0.05){$\alpha_2$}
  \caption{Proof of Proposition \ref{towersandgropes}, Step 2: Surfaces connecting the meridian and longitude of $\Sigma$ to pushoffs of the meridian of the first-level dotted circle.}\label{towersandgropesfig3}
  \end{center}
\end{figure}

\begin{figure}[ht!]
  \begin{center}
  \includegraphics[width=5in]{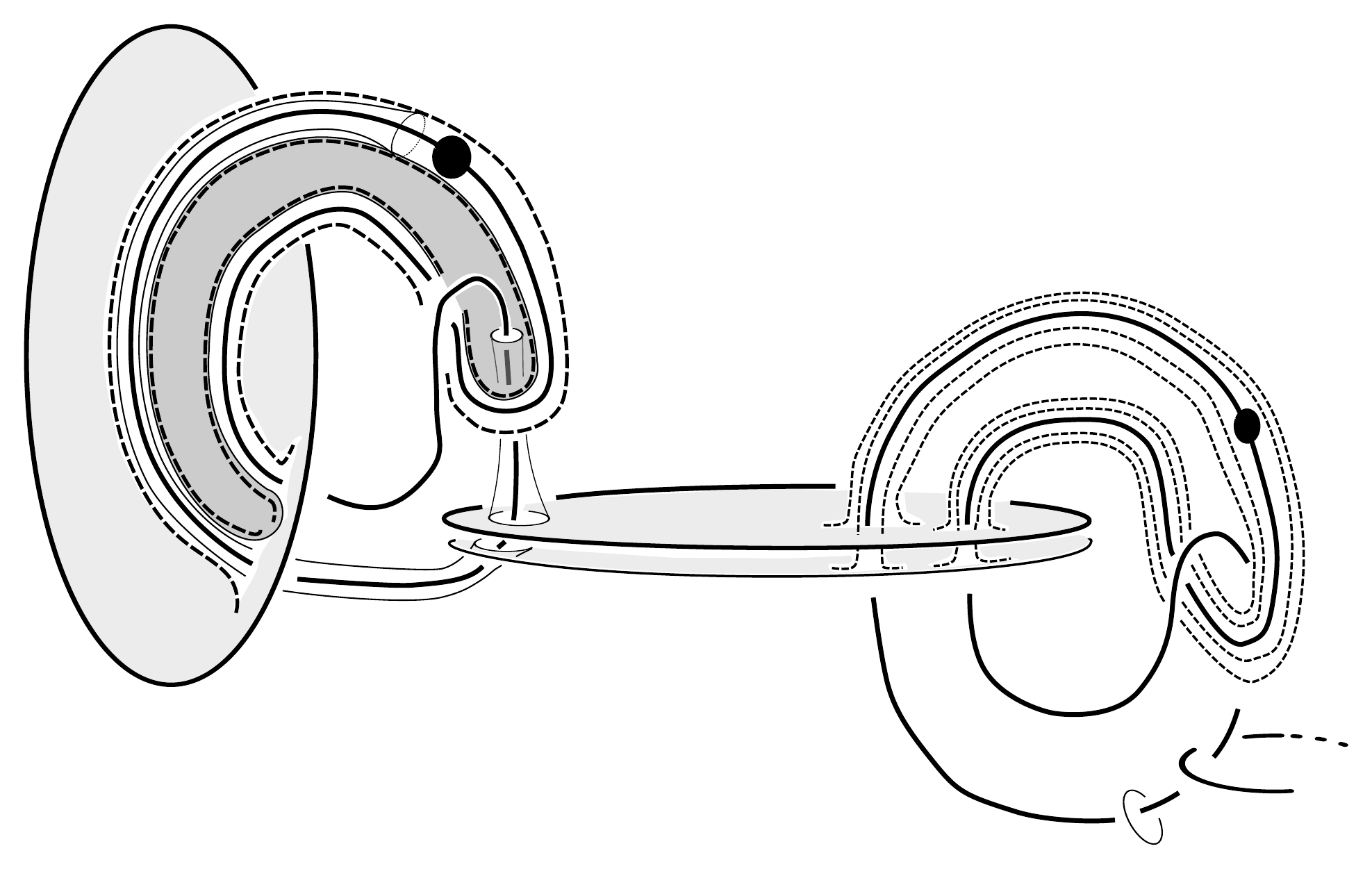}
  \put(-2.5,1.45){0}
  \put(-0.4,0.05){0}   
  \put(-0.65,-0.05){$\alpha_2$} 
  \caption{Proof of Proposition \ref{towersandgropes}, Step 3: The second stage surfaces of the grope use the annuli constructed previously (in Figure \ref{towersandgropesfig3}) and the 0--framed 2--handle attached to the meridian of the first stage dotted circle. We use two copies of the core of the attached 2--handle in addition to the standard disk shaded gray in the picture, with a tube about the second stage dotted circle. Notice that two copies of the tubes are needed and they are nested.}\label{towersandgropesfig4}
  \end{center}
\end{figure}

\begin{proof}A simple case is pictured in Figure \ref{towersandgropesfig1}, showing a neighborhood of the first two stages of a Casson tower with a single kink in each stage. We will directly and explicitly construct a grope bounded by the attaching curve (the leftmost (undecorated) curve) in Figure \ref{towersandgropesfig1}; after completing the proof in the simple case, we will outline the proof in the general case. This is partly to avoid drowning the reader in a sea of subscripts and because passing to the general case will not be particularly onerous. For clarity, we break up the proof into steps. 

\vspace{5pt}\noindent \textbf{Step 1}: The first stage of the grope, $\Sigma$, bounded by the attaching curve in the simple example, is shown in Figure \ref{towersandgropesfig2}. It  consists of the standard disk bounded by the attaching curve with a tube (dashed) along the dotted circle corresponding to the single kink in the first-stage kinky handle. 

\vspace{5pt}\noindent \textbf{Step 2}: It is easy to see, abstractly, that both the meridian $m$ and the longitude $\ell$ of $\Sigma$ are homotopic to $\alpha_1$, the meridian of the dotted circle. We easily tube `inside $\Sigma$' from $m$ to $\alpha_1$, as shown in Figure \ref{towersandgropesfig3}. We also see an embedded annulus, shown in Figure \ref{towersandgropesfig3}, cobounded by $\ell$ and a pushoff of $\alpha_1$. These two annuli intersect exactly once (as desired) at the point of intersection of $m$ and $\ell$. 

\vspace{5pt}\noindent \textbf{Step 3}: The curve $\alpha_1$ and a pushoff of $\alpha_1$ bound disjoint surfaces in the complement of $\Sigma$ and the annuli from Step 2, as follows. Each surface consists of the core (or a pushoff of the core) of the attached 0--framed 2--handle tubed along the next dotted circle, as shown in Figure \ref{towersandgropesfig4}. Since the 2--handle is attached with 0--framing, the pushoffs do not intersect. These surfaces, along with the annuli between $m$ and $\alpha_1$, and $\ell$ and $\alpha_1$, form the second stage of our grope. Note that each of the two second-stage surfaces has genus one. Call these surfaces $\Sigma_1$ and $\Sigma_2$. 

\begin{figure}[ht!]
  \begin{center}
  \includegraphics[width=3in]{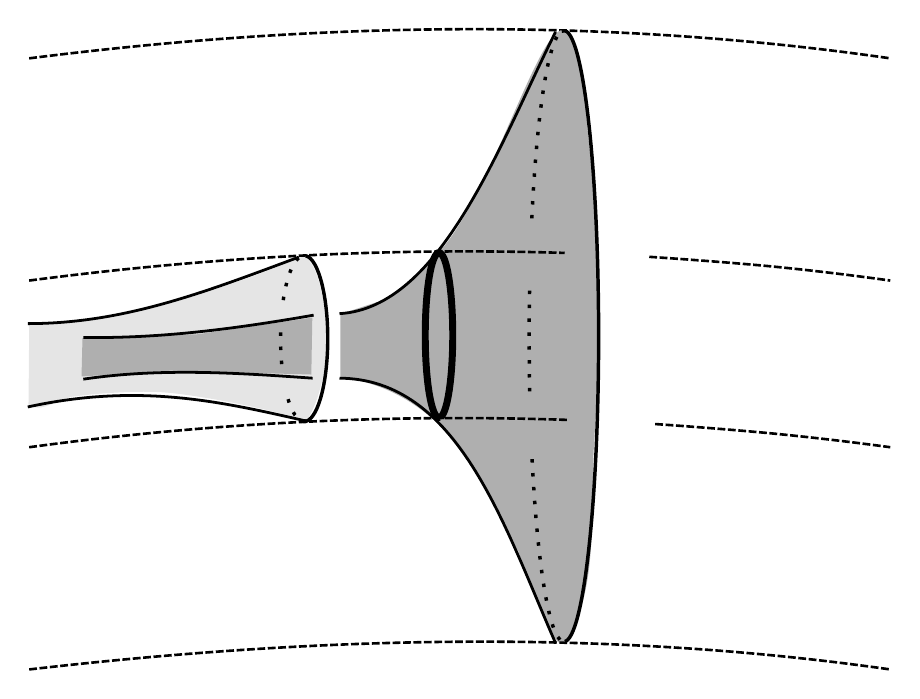}
  \put(-2.1,1.5){$m_1$}
   \put(-1.2,0.05){$m_2$} 
  \put(-0.4,0.7){$\Sigma_1$}
  \put(-0.4,0.0){$\Sigma_2$}   
  \caption{Proof of Proposition \ref{towersandgropes}, Step 4: The meridians $m_1$ and $m_2$ of the second-stage surfaces $\Sigma_1$ and $\Sigma_2$ cobound disjoint annuli $A_1$ and $A_2$ (shown in two shades of grey) with pushoffs of $\alpha_2$; however, $A_2$ intersects $\Sigma_1$ in a circle, shown in bold. We can resolve this by pushing in a neighborhood of the bolded circle into the 4--ball.}\label{towersandgropesfig5}
  \end{center}
\end{figure}

\vspace{5pt}\noindent \textbf{Step 4}: Constructing the third stage surfaces of the grope will indicate how to proceed in subsequent stages. Unlike before, we now have two sets of meridians and longitudes which are each abstractly homotopic to the meridian of the second dotted circle, $\alpha_2$. We will construct disjoint annuli cobounded by these curves and pushoffs of $\alpha_2$, away from the surfaces of the first two stages. If we proceed as we did in Step 2 we do obtain annuli that are disjoint from each other, but since the second stage surfaces are `nested', two of the four annuli intersect the second stage surfaces. However, these intersections are particularly nice---they are boundary-parallel circles in the annuli. We can push these intersections into the 4--ball to get disjoint annuli. (Here is a good toy analogy. Consider two nested, standard, unknotted tori in $S^3$. Any meridional disk of the outer torus will intersect the inner torus in a circle, but we can push the disk into the 4--ball in a neighborhood of the circle to get a meridional disk for the outer torus which is disjoint from the inner torus and still `mostly' in $S^3$.) Figure \ref{towersandgropesfig5} shows the case for the meridians of the nested tori. Let $m_2$ be the meridian of the outer surface $\Sigma_2$ and $m_1$ the meridian of the inner surface $\Sigma_1$. The tubes shown in two shades of gray in Figure \ref{towersandgropesfig5} are analogous to the tube between $m$ and $\alpha_1$ in Step 2 (Figure \ref{towersandgropesfig3}). The bolded circle (which is a meridian of $\Sigma_1$) is the intersection we need to resolve; we do so by pushing a neighborhood of it into the 4--ball. 

By this pushing in process, we obtain four embedded annuli as needed; each such annulus has a pushoff of $\alpha_2$ as one of its boundary components. Since $\alpha_2$ and its pushoffs bound disjoint surfaces as in Step 3, we can finish constructing the third stage surfaces as before. Again, note that each of the four third-stage surfaces has genus one. 

\vspace{5pt}\noindent \textbf{Step 5}: To construct the higher-stage surfaces of the grope, we essentially repeat Step 4, as follows. At the end of Step 4, we obtained four nested third-stage surfaces. As a result, to raise the grope height to four we need to find disjoint surfaces bounded by four meridian--longitude pairs. Construct annuli as before which are disjoint from one another but intersect the third-stage surfaces; these intersections are the same type as in Step 3, and we eliminate them by pushing into the 4--ball. Each such annulus has a push off of the meridian of the next dotted circle (not pictured) as one of its boundary components. These pushoffs bound disjoint surfaces as in Step 3, and therefore, we obtain eight fourth-stage surfaces, each of genus one. To construct the $n^\text{th}$--stage surfaces, we start with $2^{n-2}$ meridian--longitude pairs, and we proceed as in Step 4 to construct $2^{n-1}$ $n$--stage surfaces. by resolving $2^{n-1}(2^{n-1}-1)$ intersections by pushing into the 4--ball. 

It is easy to see, since most of the grope is in 3--dimensional space, that the attaching curve bounds this grope with untwisted framing. This finishes the proof in the simple case pictured in Figure \ref{towersandgropesfig1}

\vspace{5pt} Now we address the general case of a more complicated Casson tower. (The reader might refer to Figure \ref{height2tower} to recall the general picture). As in Figure \ref{towersandgropesfig1}, the attaching curve will be unknotted; however, the number of dotted circles linking with the attaching curve will be equal to the number of kinks in the first-stage kinky handle (of course, each pairwise linking number is zero---each dotted circle forms the Whitehead link with the attaching curve). To obtain the first stage surface of the promised grope, as we did in Step 1 above we take the standard disk bounded by the unknotted attaching curve and tube along each of the dotted circles; the resulting surface has genus equal to the number of kinks in the first-stage kinky handle. Now we must build the subsequent stages of the grope. Note that each member of a meridian--longitude pair in the first-stage surface was obtained from a specific dotted circle for the first-stage kinky handle, and as in the simple case already proved we can find an embedded annulus from each curve to the meridian of the corresponding dotted circle. Following along with the proof of the simple case, we need to find disjoint surfaces bounded by meridians of the dotted circles and their pushoffs away from the previous stage. The only change we have to make to our strategy from before is that we tube along multiple dotted circles instead of just one; however, we can do this since the dotted circles do not interact with one another. We can then continue to build all the subsequent stages using the same strategy as in the proof of the simple case. Therefore, the genera of the later-stage surfaces are equal to the number of kinks in the corresponding kinky handle in the Casson tower. (Note however that we lose the information about the signs of the kinks when going from a Casson tower to a grope.) \end{proof}
 
The following corollary is immediate. 

\begin{corollary}For each $n\geq 1$, $\mathfrak{C}_{n} \subseteq \mathcal{G}_{n}$. \end{corollary}

\begin{corollary}\label{gropeintersection}Let $\mathcal{T}$ denote the set of all topologically slice knots. Then, $$\mathcal{T}\subseteq\bigcap_{n=1}^\infty\mathcal{G}_n.$$ \end{corollary}
\begin{proof}This follows immediately from Proposition \ref{towersandgropes} and Quinn's result that any topological slice disk for a topologically slice knot contains a Casson handle \cite[Proposition 2.2.4]{Q82}\cite[Theorem 5.2]{Gom05}.\end{proof}

The above was previously known (without using Casson handles). Briefly, a topological slice disk for a knot $K$ is a topologically embedded locally flat grope of arbitrary height. Such a grope can be deformed to yield a smooth grope of arbitrary height (some more detail may be found in \cite[Remark 2.19]{Cha14}).

It is easy to see that $\mathfrak{C}_{n+2}\subseteq \mathfrak{C}_{2,\,n}$ and $\mathfrak{C}_{n+2}^\pm\subseteq \mathfrak{C}_{2,\,n}^\pm$ for all $n\geq 0$. This is because each member of a standard set of curves for the second stage of a Casson tower of height $n+2$ bounds a Casson tower of height $n$ away from the first two stages. Therefore, by Proposition \ref{towersandgropes}, each such curve bounds a grope of height $n$ away from the first two stages. In fact, a much stronger result is known, as we see below.  

\begin{corollary}\label{height3again}$\mathfrak{C}_3\subseteq \mathfrak{C}_{2,\,n}$ for all $n$. Similarly, $\mathfrak{C}_3^+\subseteq \mathfrak{C}_{2,\,n}^+$ and $\mathfrak{C}_3^-\subseteq \mathfrak{C}_{2,\,n}^-$ for all $n$.  \end{corollary}
\begin{proof} Suppose a knot bounds a Casson tower $T$ of height three. Each member of a standard set of curves for the second stage of $T$ bounds a kinky disk away from $C$, the core of the first two stages. Therefore, the curves must be null-homotopic away from $C$ and as a result, contained in $\pi_1(B^4 -  C)^{(n)}$ for all $n$.\end{proof}

\begin{proposition} $\mathfrak{C}_{2,\,n}\subseteq \mathcal{G}_{2,\,n} \subseteq \mathcal{F}_n$ for all $n\geq 0$.\end{proposition}
\begin{proof}The second inclusion is from Theorem \ref{g2n}. For the first inclusion, suppose we have a knot $K$ in $\mathfrak{C}_{2,\,n}$. That is, $K$ bounds a Casson tower $T\subseteq B^4$ of height two such that the standard set of curves are in $\pi_1(B^4 -  C)^{(n)}$, where $C$ is the core of $T$. By Proposition \ref{towersandgropes}, we know that $K$ bounds a grope $G$ of height two within $T$. In fact, we see that the generators of the first homology groups of the second stage surfaces for $G$ are exactly the meridians of the dotted circles of the second stage kinky disks of $T$, i.e.\ they are exactly the standard set of curves for $T$, which are given to be in $\pi_1(B^4 -  C)^{(n)}$. Therefore, $K\in\mathcal{G}_{2,\,n}$. \end{proof}

\begin{proposition}\label{easytower} $\mathfrak{C}_{n+2}^+ \subseteq \mathcal{P}_n$ for all $n\geq 0$. Similarly, $\mathfrak{C}_{n+2}^-\subseteq\mathcal{N}_n$ for all $n\geq 0$.\end{proposition}
\begin{proof}As before, we first show the proof in the case where there is a single kink at each stage of the Casson tower, and then describe the necessary changes for the general case. 

\begin{figure}[ht!]
  \begin{center}
  \includegraphics[width=4in]{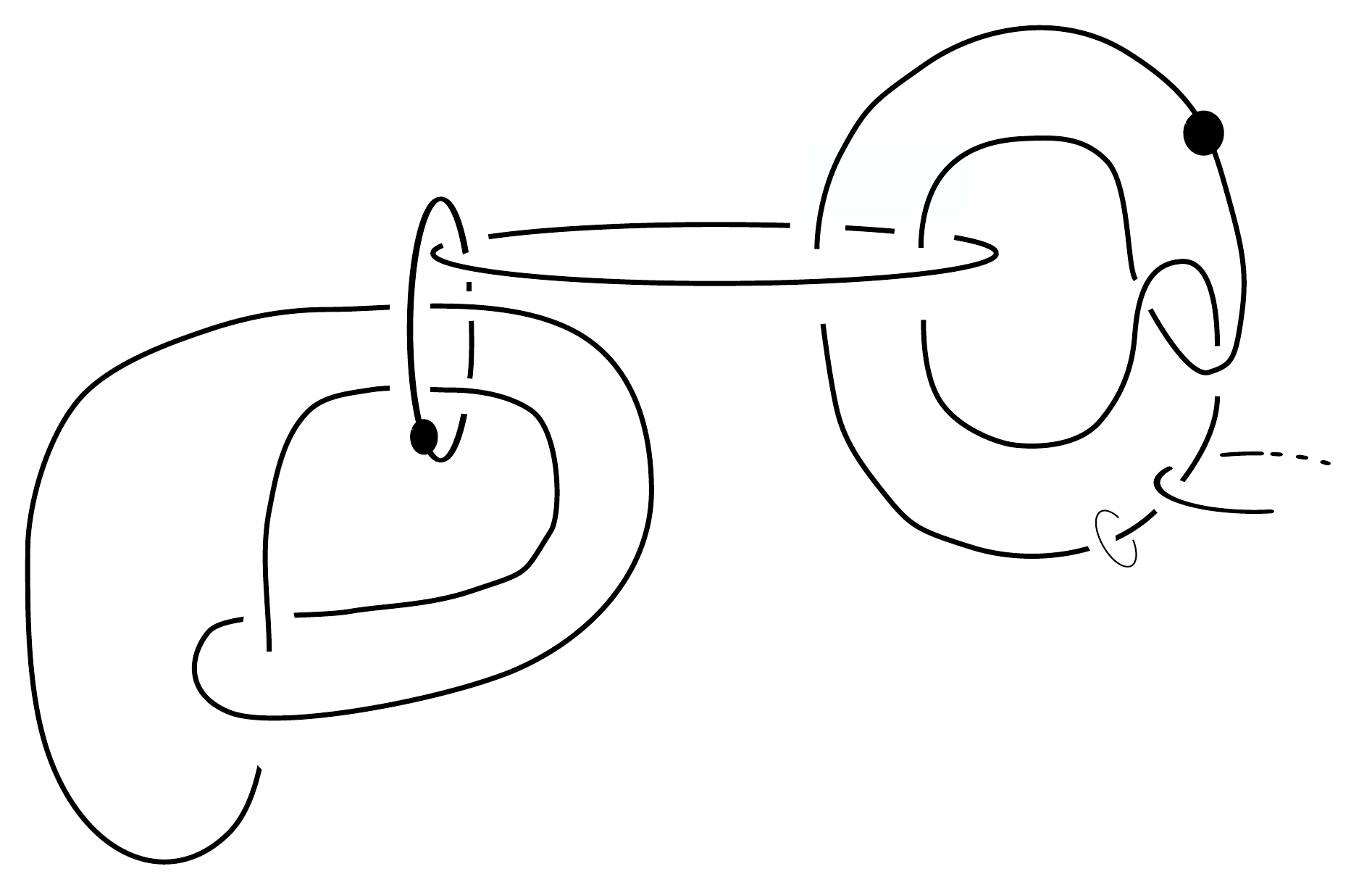}
  \put(-2,2.05){0}
  \put(-0.3,0.9){0}  
  \put(-0.75,0.8){$\alpha$}  
  \caption{Proof of Proposition \ref{easytower}: A Kirby diagram of the first two stages of a Casson tower with a single positive kink at each stage.}\label{proof1}
  \end{center}
\end{figure}

\begin{figure}[ht!]
  \begin{center}
  \includegraphics[width=4in]{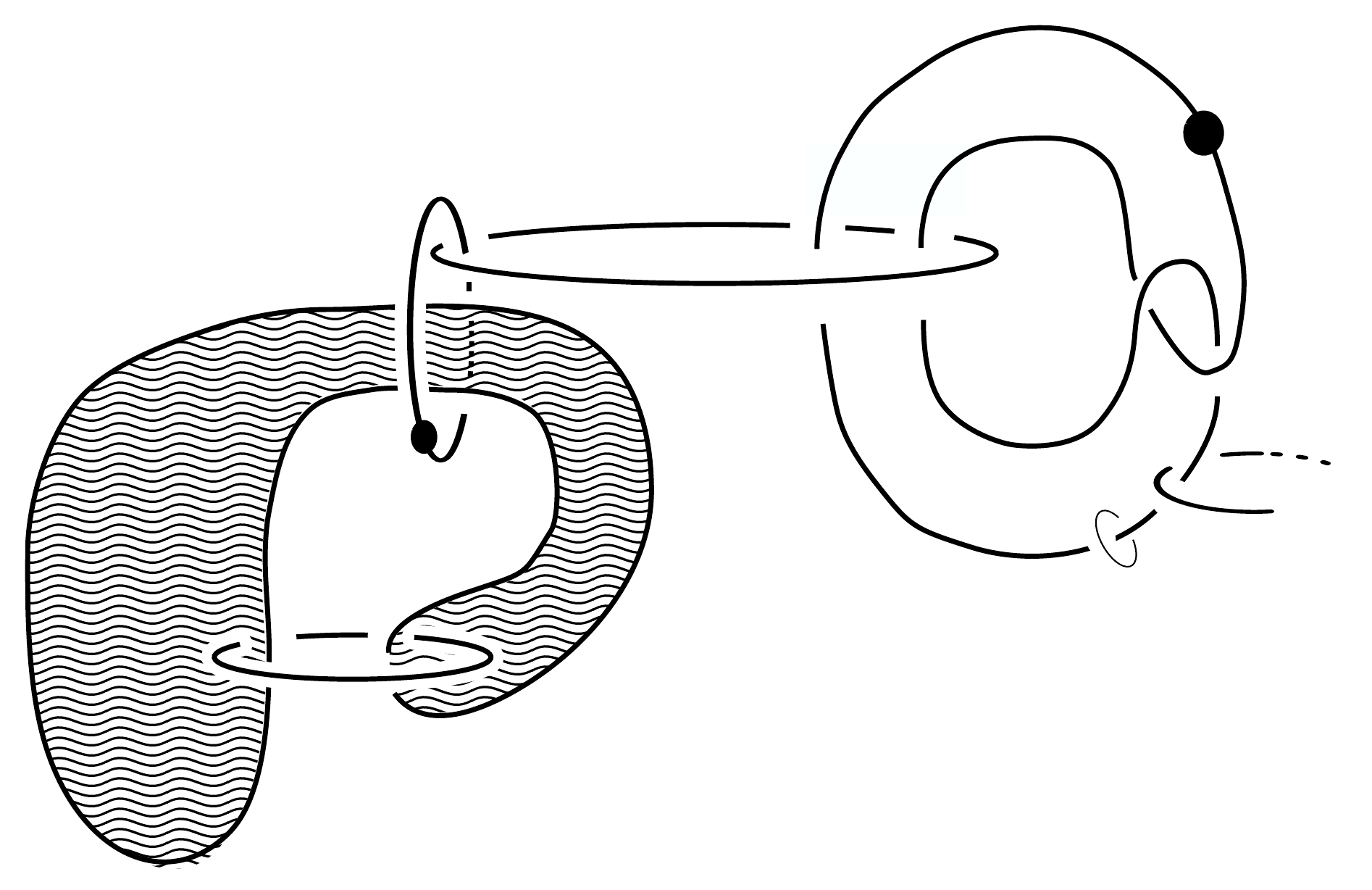}
  \put(-2,2.05){0}
  \put(-0.3,0.9){0}  
  \put(-0.75,0.8){$\alpha$}    
  \put(-3.2,0.4){+1} 
  \caption{Proof of Proposition \ref{easytower}: After blowing up the first stage kink, a slice disk for the attaching curve can be seen.}\label{proof2}
  \end{center}
\end{figure}

\begin{figure}[ht!]
  \begin{center}
  \includegraphics[width=4in]{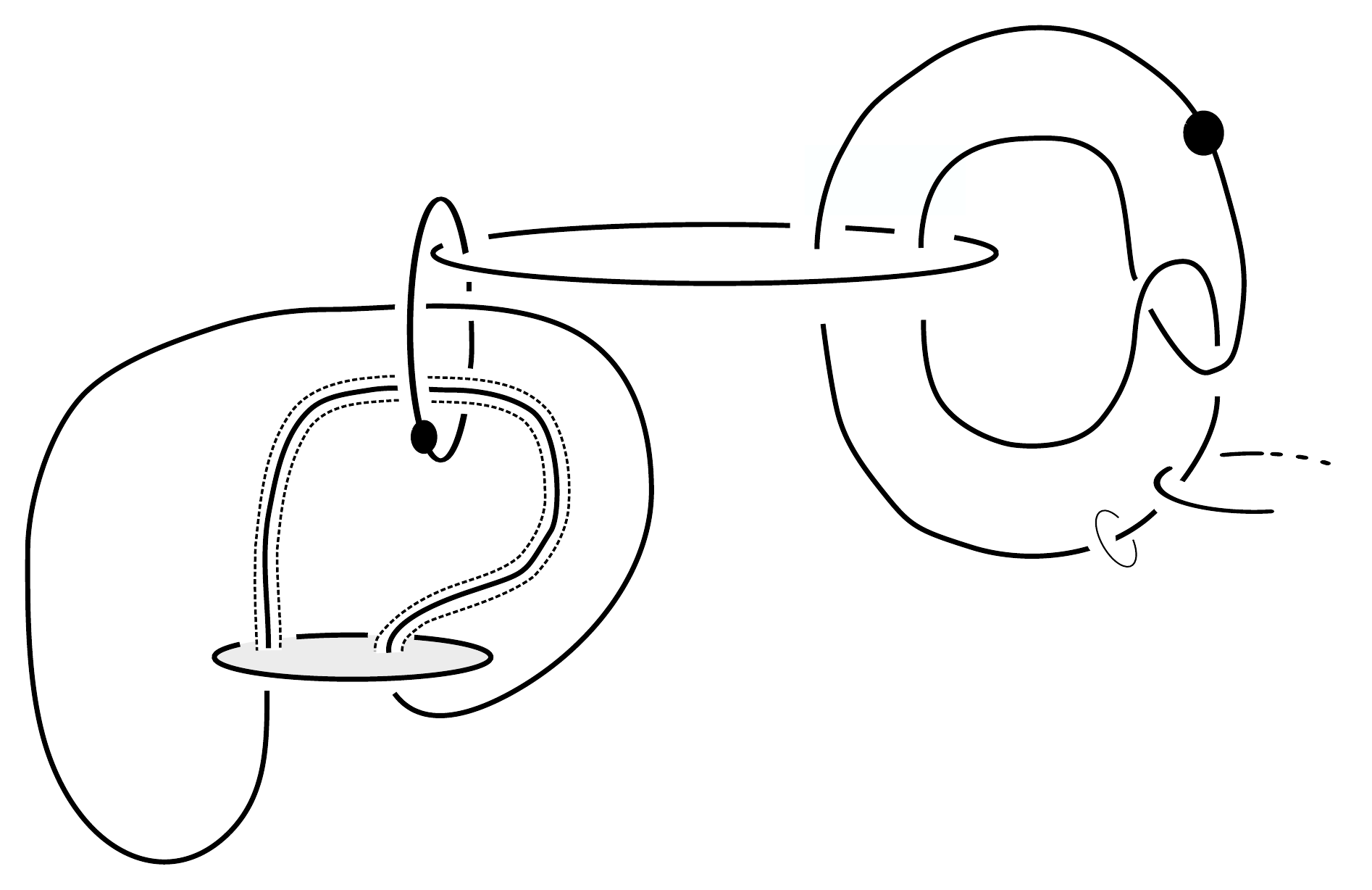}
  \put(-2,2.05){0}
  \put(-0.3,0.9){0}  
  \put(-0.75,0.8){$\alpha$}   
  \put(-3.2,0.4){+1}    
  \caption{Proof of Proposition \ref{easytower}: Finding generators of $H_2(V)$.}\label{h2generators1}
  \end{center}
\end{figure}

\begin{figure}[ht!]
  \begin{center}
  \includegraphics[width=4in]{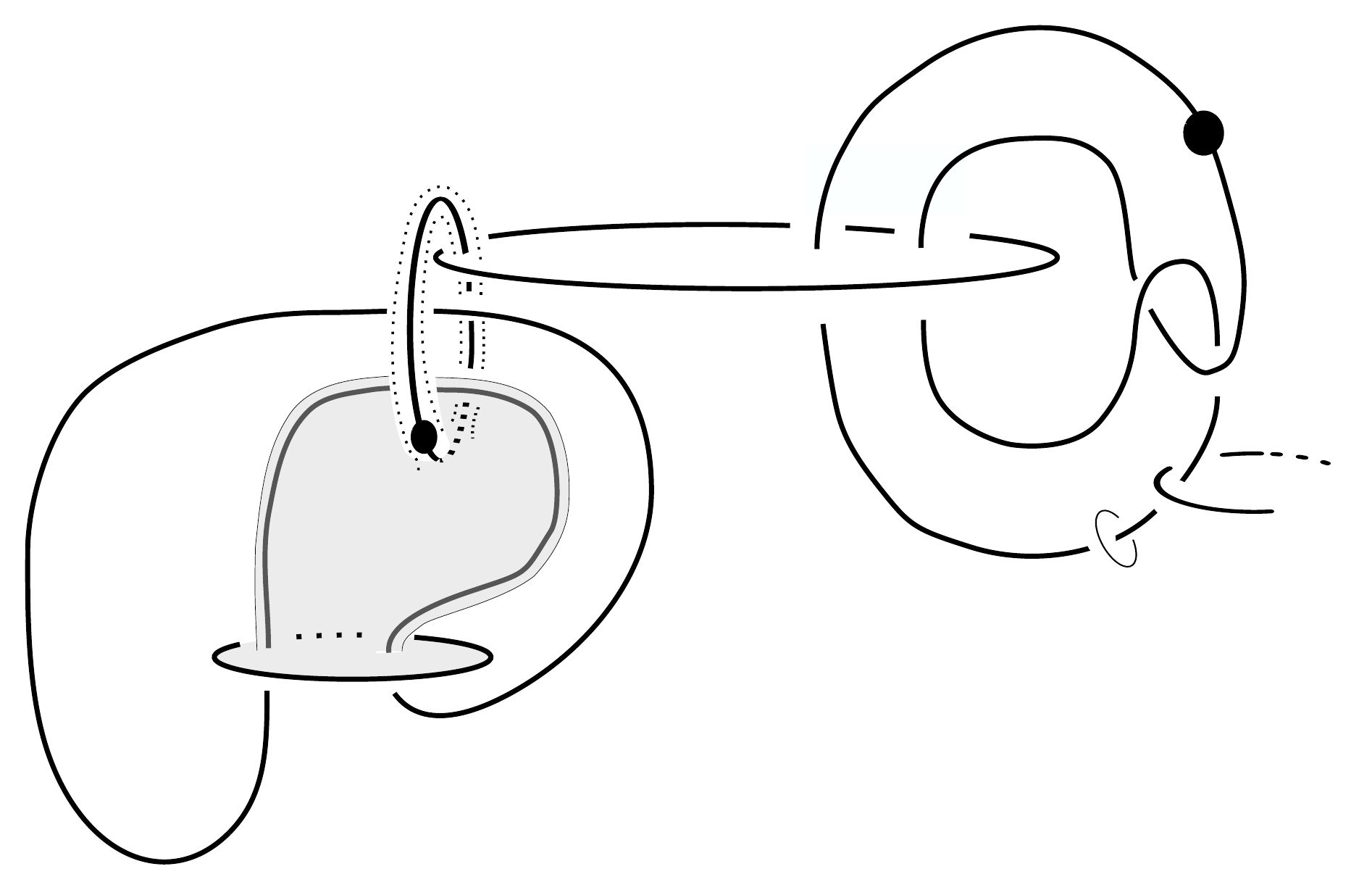}
  \put(-2,2.05){0}
  \put(-0.3,0.9){0}  
  \put(-0.75,0.8){$\alpha$}   
  \put(-3.2,0.4){+1}      
  \caption{Proof of Proposition \ref{easytower}: Finding generators of $H_2(V)$.}\label{h2generators2}
  \end{center}
\end{figure}

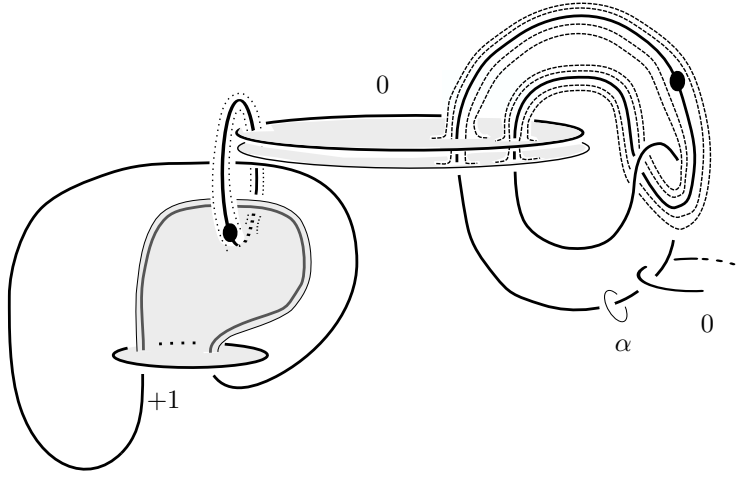
\begin{figure}[ht!]
  \begin{center}
  \includegraphics[width=4in]{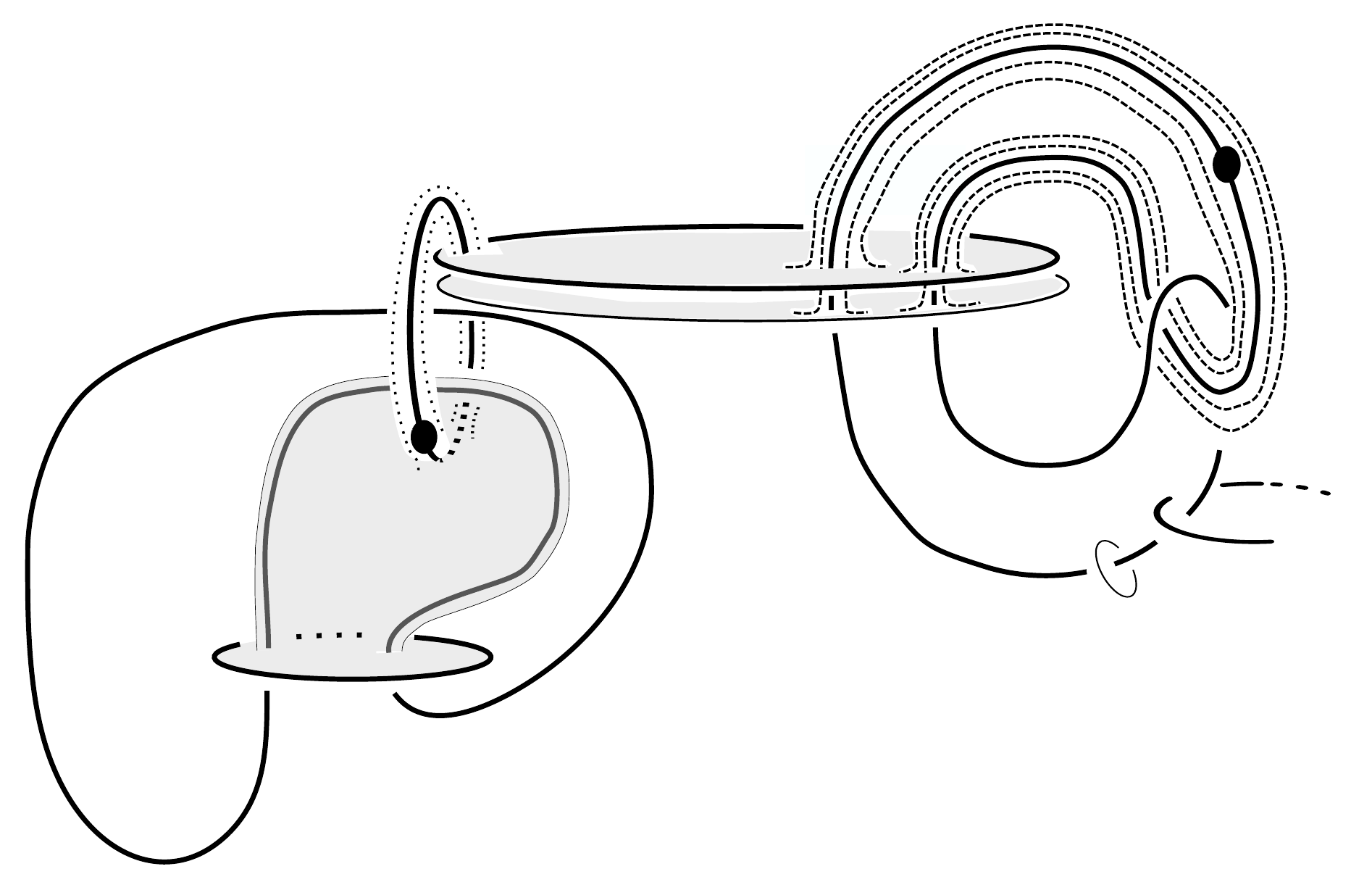}
  \put(-2,2.05){0}
  \put(-0.3,0.8){0}  
  \put(-0.75,0.7){$\alpha$}    
  \put(-3.2,0.4){+1}   
  \caption{Proof of Proposition \ref{easytower}: Finding generators of $H_2(V)$.}\label{h2generators3}
  \end{center}
\end{figure}

Figure \ref{proof1} shows a Kirby diagram for the first two stages of a Casson tower $T$ with a single positive kink at each stage. We blow up at the kink in the first stage disk. In our Kirby diagram, Figure \ref{proof2}, this introduces a $+1$--framed 2--handle, indicating that the new manifold is diffeomorphic to $T \#\,\mathbb{CP}(2)$. Since the blow up occurred in the interior of $T$, we have an embedding $T\#\,\mathbb{CP}(2)\hookrightarrow B^4\#\,\mathbb{CP}(2)$ (where earlier we had an embedding $T\hookrightarrow B^4$). Let $V$ denote $B^4\#\,\mathbb{CP}(2)$. Recall that a 0--framed neighborhood of the knot $K$ can be seen as a 0--framed neighborhood of the undecorated attaching curve in the Kirby diagram for the Casson tower. As a result, a slice disk $\Delta$ for $K$ is obvious in $T\#\mathbb{CP}(2)$, shown in Figure \ref{proof2} (the attaching curve for the 2--handle pierces through it twice transversely). We also see that the 0--framed push off of the Casson tower attaching curve bounds a disjoint parallel disk. Since $V$ is simply connected with positive definite intersection form all that remains to be done is to find a generator $S$ for $H_2(V)\cong\mathbb{Z}$ such that $\pi_1(S)\subseteq \pi_1(V - \Delta)^{(n)}$. We will do so by finding a generator $S$ such that the generators of $\pi_1(S)$ bound gropes of height $n$ in $(T\#\,\mathbb{CP}(2)) -  \Delta$.

For clarity, we describe how we obtain such an $S$ in several steps. The na\"{\i}ve choice of generator for $H_2(V)$ is the core of the attached $+1$--framed 2--handle along with the obvious disk bounded by it, shaded in gray in Figure \ref{h2generators1}. However, this clearly intersects $\Delta$. We can avert this problem by tubing along the attaching curve. While this does yield a torus generator for $H_2(V)$ disjoint from $\Delta$, one of its $H_1$--generators is the meridian of the attaching curve (and therefore the knot.) We try to fix this by surgering along the longitude of the torus using the obvious disks (pierced through by the dotted circle in Figure \ref{h2generators1}). The 2--sphere obtained intersects the dotted circle and so we tube along it, as shown in Figure \ref{h2generators2}. This yields another torus generator for $H_2(V)$, but once again, one of its $H_1$--generators is the meridian of the attaching curve. Fortunately, we can address this easily by noting that the meridian of the present torus bounds a punctured torus. Cut along the meridian and glue in two copies of the punctured torus to finally obtain a generator $S$ of $H_2(V)$ in Figure \ref{h2generators3}. We claim that this is the desired surface generating $H_2(V)$. 

Note that each member of the standard generating set for $\pi_1(S)$ is homotopic to a meridian of the second stage dotted circle, $\alpha$ (i.e.\ the standard curve for the second stage of $T$) away from $\Delta$. Since the standard curve bounds a Casson tower of height $n$ (and therefore a grope of height $n$) away from $\Delta$, we see that $\pi_1(S)\subseteq\pi_1(V - \Delta)^{(n)}$. But we can do better---we can show that the members of the standard generating set for $\pi_1(S)$ themselves bound disjoint gropes away from the first two stages of $T$. The only additional step needed is to find disjoint annuli connecting the generators of $\pi_1(S)$ to $\alpha$. This is the same construction as in the proof of Proposition \ref{towersandgropes} when we constructed the third stage surfaces of a grope, and we omit it to avoid repetition.

The reader might ask why $S$ constructed above is a generator of $H_2(V)$. To see this, start with the na\"{\i}ve choice of generator $s$, namely, the standard disk (shaded in gray in Figure \ref{h2generators1}) capped off with the core of the +1--framed 2--handle. Take a pushoff $\bar{s}$ of $s$. The spheres $\bar{s}$ and $s$ intersect exactly once transversely with positive sign. Now perform the various tubing operations described above on $\bar{s}$---we can do so in the complement of $s$. The resulting surface $S$ will continue to have a single positive transverse intersection with $s$ and therefore is in the same homology class as $s$. 

For more complicated Casson towers, we apply the same process. To begin with, in the Kirby diagram for a general Casson tower, the attaching curve will clasp itself multiple times (equal to the number of kinks in the first-stage kinky handle), as shown in Figure \ref{height2tower}. We blow up at each of these kinks, introducing a +1--framed 2--handle at each clasp. Since the clasps are isolated from one another the first step of the proof in the simple case goes through and we can observe a slice disk for the attaching curve which is pieced through twice transversely by each introduced 2--handle. The number of generators of $H_2(V)$ needed is equal to the number of kinks in the first stage of the tower, since this is equal to the number of $\CP(2)$ summands introduced to the 4--manifold. We construct these generators as before. In particular, at each +1--framed 2--handle, we start with the core of the 2--handle tubed along the attaching curve, surger along the longitude, and tube along the dotted circle (Figures \ref{h2generators1} and \ref{h2generators2}). Since the newly introduced 2--handles do not interact with each other, we can do this just as easily as in the simple case. The only difference for the general case comes in the next step, pictured in Figure \ref{h2generators3}. In the general case we might have multiple kinks in the second stage kinky handles, each of whom will contribute a dotted circle to the Kirby diagram. As we did in the proof of Proposition \ref{towersandgropes}, we simply tube along each such dotted circle to get the desired generators of $H_2(V)$. Since Proposition \ref{towersandgropes} holds for general Casson towers, this completes the proof in the general case. For each kink in the first stage, the genus of the corresponding member of the set of generators of $H_2(V)$ is equal to the number of kinks in the associated second stage kinky disk. 

The above shows that $\mathfrak{C}_{n+2}^+\subseteq\mathcal{P}_n$. For a knot $K\in\mathfrak{C}_{n+2}^-$ the kinks in the first stage kinky disk would be negative and we would blow up using $-1$--framed 2--handles, indicating a connected sum with $\overline{\mathbb{CP}(2)}$. The rest of the construction is analogous.\end{proof}

\begin{proposition}\label{weirdtower} $\mathfrak{C}_{2,\,n}^+ \subseteq \mathcal{P}_n$ for all $n\geq 0$. Similarly, $\mathfrak{C}_{2,\,n}^-\subseteq \mathcal{N}_n$ for all $n\geq 0$.\end{proposition}
\begin{proof}The proof of our previous proposition did not truly require the Casson tower beyond the first two levels. If the standard set of curves of a tower of height two is known to be in the $n^\text{th}$--derived subgroup of the fundamental group of the exterior of the core of the first two stages, the remainder of the proof follows identically. \end{proof} 

The following is now an immediate corollary of Corollary \ref{height3again} and and the above result, and reveals the inefficacy of Proposition \ref{easytower} in studying the positive and negative filtrations of $\mathcal{C}$.

\begin{corollary}\label{height3}$\mathfrak{C}_3^+ \subseteq \bigcap_{n=0}^\infty \mathcal{P}_n$. Similarly, $\mathfrak{C}_3^- \subseteq \bigcap_{n=0}^\infty \mathcal{N}_n$. \end{corollary}

The results of this section constitute the various pieces of Theorem A.


\section{Characterization of knots in $\mathfrak{C}_1^\pm$}\label{c1+}

Knots in $\mathfrak{C}_1^\pm$ can be completely characterized by the following theorem. 

\begin{thm_B}For any knot $K$, the following statements are equivalent. 
\begin{itemize}
\item[(a)] $K\in \mathfrak{C}_1^+$ (resp. $\mathfrak{C}_1^-$)
\item[(b)] $K$ is concordant to a fusion knot of split positive (resp. negative) Hopf links
\item[(c)] $K$ is concordant to a knot which can be changed to a ribbon knot by changing only positive (resp. negative) crossings.
\end{itemize}
\end{thm_B}

\begin{remark}In \cite[Remark 3.3, Lemma 3.4]{CLic86}, Cochran--Lickorish showed that if a knot can be changed to the unknot by only changing positive (resp. negative) crossings, it bounds a kinky disk in the 4--ball with only positive (resp. negative) kinks---very little further insight is needed to prove the more general statement $(c)\Rightarrow (a)$. We include it here for completeness. 

This result should also be compared with a characterization of knots in a particular subset of $\mathcal{P}_0$ given by Cochran--Tweedy in \cite{CTwee13}. \end{remark}

\begin{figure}[ht!]
  \begin{center}
  \includegraphics{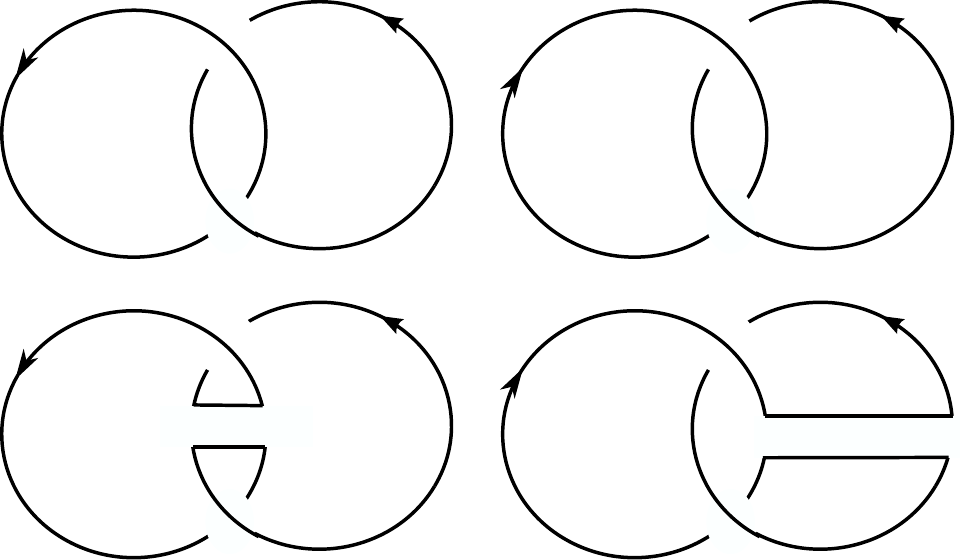}
  \caption{Both the positive (left) and the negative (right) Hopf link can be fused to yield the unknot.}\label{hopffusion}
  \end{center}
\end{figure}

\begin{proof}[Proof of Theorem B]Suppose $K\in \mathfrak{C}_1^+$, i.e.\ $K$ bounds a kinky disk $\overline{\Delta}$ in $B^4$ with all kinks positive. As before we blow up each kink of $\overline{\Delta}$ with a $\CP(2)$ to resolve the singularities of $\overline{\Delta}$. Remove a tubular neighborhood of the core $\CP(1)$ within each added $\CP(2)$. This results in a number of additional $S^3$ boundary components which intersect $\overline{\Delta}$ in positive Hopf links. We can tube these newly created $S^3$'s together. Since the tube acts like a 1--dimensional submanifold of a 4--manifold, it may be considered to be disjoint from $\overline{\Delta}$. We excise the tube; the resulting 4--manifold $W$ is diffeomorphic to $S^3 \times [0,1]$ where $K$ is contained in the $S^3\times \{1\}$ boundary component. By throwing away any additional components, we get a smooth genus zero surface $\Delta$ cobounded by $K$ and a split collection of positive Hopf links. (The Hopf links are split in the sense that they can be separated from one another by a collection of disjoint, smoothly embedded 2--spheres.) 

By an isotopy relative to the boundary, we can ensure that the height function on $W\cong S^3\times [0,1]$ restricts to a Morse function on $\Delta$ and that the maxima occur at the $t=\frac{4}{5}$ level, the join saddles at the $t=\frac{3}{5}$ level, the split saddles at $t=\frac{2}{5}$ and the minima at $t=\frac{1}{5}$. The intersection of $\Delta$ with $t=\frac{1}{2}$ is then a connected 1--manifold embedded in $S^3\times \{\frac{1}{2}\} \cong S^3$. Call this knot $J$. The portion of $\Delta$ in $S^3\times [\frac{1}{2},1]$ gives a concordance between $K$ and $J$. We will show that $J$ is a fusion of split positive Hopf links.

The portion of $\Delta$ in $S^3\times[0,\frac{1}{2}]$ is almost what we need already. In particular, it demonstrates $J$ as a fusion of an unlink (from the minima of $\Delta$) and a split collection of positive Hopf links. However, each component of the unlink can be considered as a fusion of a positive Hopf link, as shown in Figure \ref{hopffusion}. To be more specific, we can use an arc disjoint from $\Delta$ to extend each minimum down to $S^3 \times \{\epsilon\}$. Since the minima form an unlink we can keep them split from one another and the Hopf links. Within $S^3\times [0,\epsilon]$, we can use saddles to split the unknotted components into positive Hopf links. This shows that $J$ is a fusion of a collection of split positive Hopf links, and therefore (a)$\Rightarrow$(b).

Now suppose that $K$ is concordant to a fusion knot of split positive Hopf links. Since a positive Hopf link can be changed to an unlink by changing a positive crossing, (b)$\Rightarrow$(c) is clear.

Suppose that $K$ is concordant to a knot which can be changed to a ribbon knot by only changing positive crossings, i.e.\ there is a kinky annulus in $S^3\times [0,1]$, with only positive kinks, cobounded by $K$ and a ribbon knot $J$. By appending the slice disk for $J$, we get a kinky disk with only positive kinks bounded by $K$ in $B^4$.

The corresponding statements for $\mathfrak{C}_1^-$ can be proved by an entirely analogous argument. \end{proof}

Using an almost entirely identical argument, we can prove the following proposition.
\begin{proposition}For any knot $K$, the following statements are equivalent.
\begin{itemize}
\item[(a)] $K$ bounds a kinky disk with $p$ positive and $n$ negative kinks.
\item[(b)] $K$ is concordant to a fusion knot of $p$ positive Hopf links, $n$ negative Hopf links and an unlink
\item[(c)] $K$ is concordant to a knot that can be changed to a ribbon knot by changing $p$ positive and $n$ negative crossings.
\end{itemize}
\end{proposition}

\subsection{Positivity of knots}
Theorem B involves several notions which might reasonably be referred to as `positivity' for knots. It is instructive to study how they are related to other such notions which are well-established in the literature. Let us start by listing some of these concepts.

\begin{enumerate}\label{positivitysection}
\item $K$ is the closure of a positive braid
\item $K$ has a projection where all crossings are positive 
\item $K$ is strongly quasipositive 
\item $K$ is quasipositive
\item $\kappa_-(K)=0$
\item $K$ bounds a kinky disk in $B^4$ with only positive kinks, i.e.\ $K\in\mathfrak{C}_1^+$
\item $K$ is concordant to a knot that can be changed to a slice knot by changing only positive crossings
\item $K$ is concordant to a fusion knot of a split collection of positive Hopf links
\item $K \in \mathcal{P}_0$
\end{enumerate}
In the list above, $\kappa_-$ denotes \textit{negative kinkiness}, a smooth concordance invariant defined by Gompf in \cite{Gom86}, which is equal to the minimum number of negative kinks in a kinky disk bounded by $K$ in the 4--ball. Similarly, the \textit{positive kinkiness} $\kappa_+$ of a knot $K$ is the least number of positive kinks in a kinky disk bounded by $K$ in the 4--ball, and the \textit{kinkiness} of $K$ is the ordered pair $(\kappa_+(K), \kappa_-(K))$.The terms quasipositivity and strong quasipositivity are due to Rudolph; see \cite{Rud05} for a thorough exposition. 
\begin{figure}[t!]
  \begin{center}
  \includegraphics[width=5in]{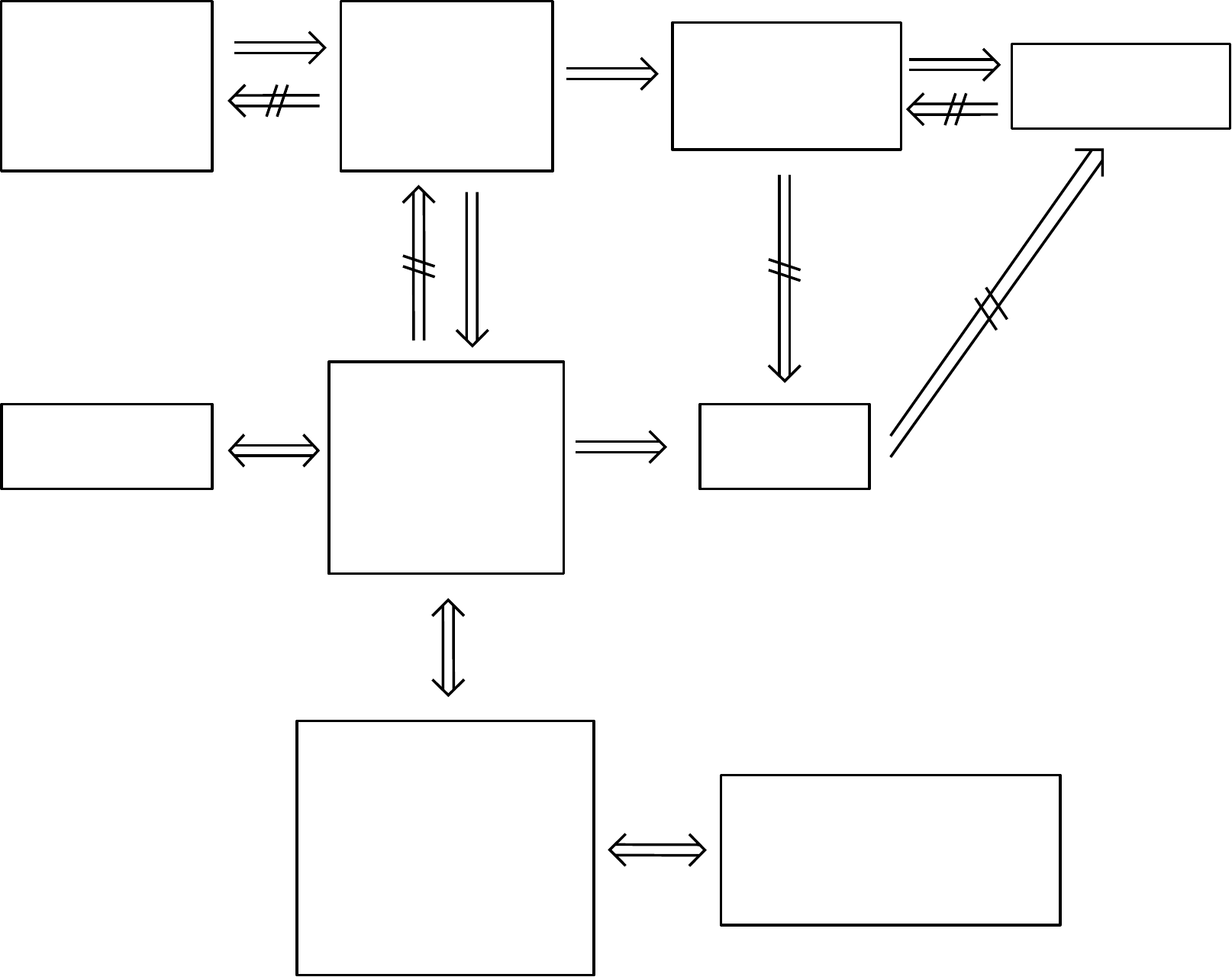}
  \put(-4.925,3.8){closure of a}
  \put(-4.8,3.6){positive}
  \put(-4.72,3.4){braid} 
  \put(-3.4,3.8){has a}
  \put(-3.45,3.6){positive}
  \put(-3.5,3.4){projection} 
  \put(-2.07,3.7){strongly}
  \put(-2.2,3.5){quasipositive}
  \put(-0.85,3.6){quasipositive}
  \put(-4.925,2.125){$\kappa_-(K)=0$}
  \put(-3.45,2.3){bounds a}
  \put(-3.5, 2.1){kinky disk}
  \put(-3.45,1.9){with only}
  \put(-3.6, 1.7){positive kinks} 
  \put(-2.075, 2.125){$K\in\mathcal{P}_0$}
  \put(-3.65,0.875){concordant to a}
  \put(-3.65,0.675){knot which can}
  \put(-3.6,0.475){be sliced by}
  \put(-3.625,0.275){changing only}
  \put(-3.725,0.075){positive crossings}
  \put(-1.85,0.68){concordant to a}
  \put(-1.8,0.48){fusion knot of}
  \put(-1.95,0.28){positive Hopf links}
  \caption{Known relationships between some notions of positivity of knots.}\label{positivity}
  \end{center}
\end{figure}

The known relationships between the above notions of positivity of knots are summarized in Figure \ref{positivity}. Item $(1)\Rightarrow$ Item $(2)$ trivially, but there are examples of knots with positive projections which are not closures of positive braids. Rudolph showed in \cite{Rud99} that knots with positive projections are strongly quasipositive. Strongly quasipositive knots are obviously quasipositive by definition. However, Baader showed in \cite{Baa05} that there exist quasipositive knots that are not strongly quasipositive\footnote{These examples were pointed out by Steven Sivek in response to a question posed by the author on MathOverflow \cite{MO}.}.

Items (5) and (6) are equivalent by the definition of $\kappa_-$ and items $(6)$, $(7)$ and $(8)$ are equivalent by Theorem B. Item $(5)$ implies item $(9)$ as discussed previously, by `blowing up' at the kinks of a kinky disk. Any knot with a positive projection bounds a kinky disk with only positive kinks; this is so since it can be unknotted by changing only positive crossings. Therefore, item $(2)$ implies item $(6)$. However, it is known that knots with positive projections necessarily have (strictly) negative signatures \cite{CGom88,Przy89,Trac88} while knots (such as the figure eight knot) with zero signature may bound kinky disks with only positive kinks. As a result, item $(6)$ does not imply item $(2)$.

Rudolph showed that one can construct a strongly quasipositive knot with any given Seifert pairing \cite{Rud83,Rud05}. This implies that we can find strongly quasipositive knots with positive signature, which obstructs membership in $\mathcal{P}_0$ by \cite[Proposition 1.2]{CHHo13}. Membership in $\mathcal{P}_0$ does not imply strong quasipositivity, or even quasipositivity. As pointed out in \cite[Remark 4.6]{Rud89}, a non-slice knot which is its own mirror image (such as the figure eight knot, which lies in $\mathcal{P}_0$) cannot be quasipositive. On the other hand, it is true that if $K$ is strongly quasipositive, then $K\notin\mathcal{N}_0$, as follows. Livingston proved in \cite{Liv04}\footnote{Livingston's result is not stated in terms of strong quasipositivity. The equivalence of Livingston's conditions and strong quasipositivity is pointed out by Hedden in the introduction to \cite{He10}} that if $K$ is strongly quasipositive, then $\text{g}(K)=\text{g}_4(K)=\tau(K)$, where $\text{g}_4$ denotes smooth 4--genus and $\tau$ denotes Ozsv\'{a}th--Szab\'{o} \cite{Os03} and Rasmussen's \cite{Ras03} smooth concordance invariant. Therefore, any non-trivial, strongly quasipositive $K$ has $\tau(K)>0$, which obstructs membership in $\mathcal{N}_0$ by \cite[Proposition 1.2]{CHHo13}. Collectively this paragraph addresses a question posed in Section 3 of \cite{CHHo13}, seeking the relationship between strong quasipositivity and $\mathcal{P}_0$.

The relationships summarized in Figure \ref{positivity} lead to the natural question of whether membership in $\mathcal{P}_0$ implies any of the equivalent notions $(5)$--$(8)$. This seems unlikely to be true, but we do not have a counterexample at present. 
\section{Examples and properties}\label{exprop}

\begin{example}\label{allc1}It is well-known that any knot can be changed to the unknot by changing crossings (the minimum number of crossings that need to be changed is the \textit{unknotting number} of a knot). By tracing the homotopy corresponding to the crossing changes, we see that \textit{every} knot bounds a kinky disk if we impose no restrictions on the signs of the kinks, i.e.\ every knot lies in $\mathfrak{C}_1$.\end{example}

\begin{example}\label{intersection}Theorem B shows that membership in $\mathfrak{C}_1^\pm$ is harder. From Proposition 1.2 in \cite{CHHo13} we know that the signs of various well-known concordance invariants obstruct membership in $\mathcal{P}_0$ and $\mathcal{N}_0$. Since $\mathfrak{C}_1^+\subseteq\mathcal{P}_0$ and $\mathfrak{C}_1^-\subseteq\mathcal{N}_0$, they also obstruct membership in $\mathfrak{C}_1^+$ and $\mathfrak{C}_1^-$. For example, if $\tau(K)<0$, $K\notin\mathcal{P}_0$, and therefore,  $K\notin\mathfrak{C}_1^+$. Similarly, $K\notin\mathfrak{C}_1^+$ if the Levine--Tristram signature of $K$ is strictly positive, or $\text{s}\,(K)<0$. Using Theorem B, we can then see that the signs of these invariants also obstruct when a knot can be changed to a slice knot by changing only positive or negative crossings. Results of this nature were proved by Cochran--Lickorish and Bohr in \cite{CLic86} and \cite{Bohr02} respectively. \end{example}

\begin{figure}[t!]
  \begin{center}
  \includegraphics[width=1.8in]{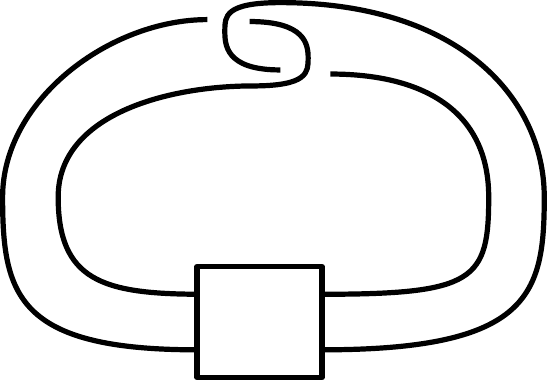}
  \put(-1,0.15){$n$}
  \caption{The knot $T_n^-$. The box with a number `$n$' inside should be interpreted as $n$ full twists.}\label{twistknots}
  \end{center}
\end{figure}

\begin{example} By Theorem B, any knot which can be changed to a slice knot by changing positive (resp. negative) crossings lies in $\mathfrak{C}_1^+$ (resp. $\mathfrak{C}_1^-$). In particular this implies that any knot with unknotting number one, or even \textit{slicing number} one, lies in either $\mathfrak{C}_1^+$ or $\mathfrak{C}_1^-$ (the slicing number of a knot is the minimum number of crossing changes needed to change it to a slice knot).

Let $T^+_n$ denote the positively-clasped twist knots with $n$ twists and $T^-_n$ the negatively-clasped twist knots (see Figure \ref{twistknots}). Clearly, each $T^\pm_n$ can be unknotted by changing one of the crossings at the clasp and therefore, $T^+_n\in\mathfrak{C}_1^+$ and $T^-_n\in\mathfrak{C}_1^-$ for all $n$. On the other hand, for positive $n$, the knot $T^+_n$ can be unknotted by changing $n$ negative crossings (undoing the $n$ twists) and therefore, $T_n^+\in\mathfrak{C}_1^+\cap\mathfrak{C}_1^-$ for positive $n$. Similarly, $T_n^-\in\mathfrak{C}_1^+\cap\mathfrak{C}_1^-$ for negative $n$. Note that it is easy to see that $T^+_n$ is a fusion of a positive Hopf link. However, since $T^+_n\in\mathfrak{C}_1^-$ for positive $n$, such a knot must also be concordant to a fusion of negative Hopf links by Theorem B---an example is shown in Figure \ref{twistknotasfusion}.\end{example}

\begin{figure}[t!]
  \begin{center}
  \includegraphics[width=4in]{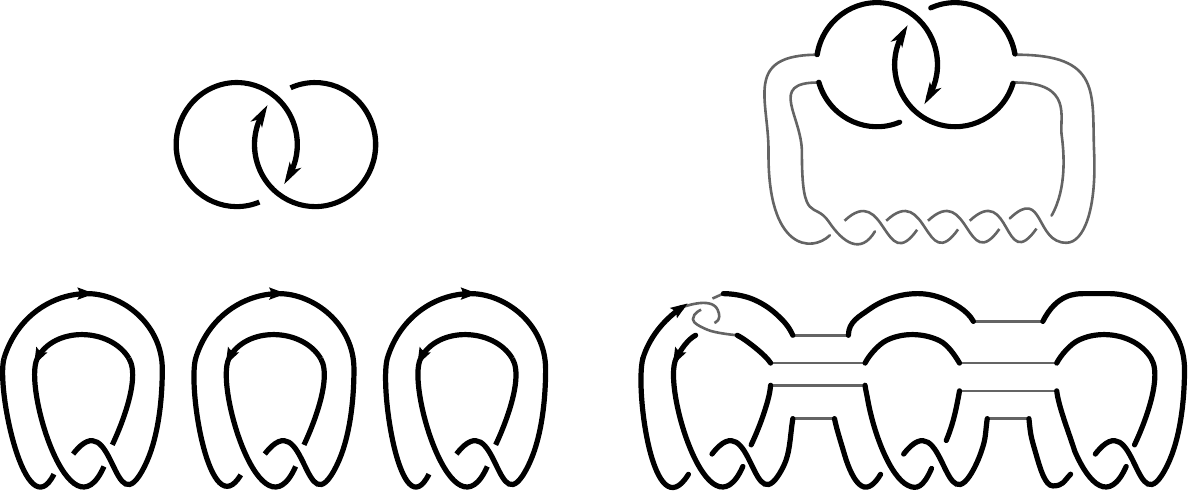}
  \caption{The knot $T^+_3$ can be obtained as a fusion of a single positive Hopf link, or as a fusion of three negative Hopf links. Fusion bands are shown in gray.}\label{twistknotasfusion}
  \end{center}
\end{figure}

\begin{example}Example 4.5 of \cite{CHHo13} shows that $\Wh[-]_0(LHT)\notin\mathcal{P}_0$, where $LHT$ is the left-handed trefoil and $\Wh[-]_0(\cdot)$ denotes the negatively clasped zero-twisted Whitehead double. Similarly Example 4.6 in \cite{CHHo13} shows that if $p<0$, $q>0$, and $r>0$ are odd and $pq+qr+rp=-1$, then the pretzel knots $K(p,q,r)\notin\mathcal{P}_0$. Therefore, since $\mathfrak{C}_1^+\subseteq\mathcal{P}_0$, none of these knots can bound a kinky disk with only positive kinks and by Theorem B none of these knots can be changed to a slice knot by changing only positive crossings.  \end{example}

Example \ref{intersection} showed that $\mathfrak{C}_1^+$ and $\mathfrak{C}_1^-$ have non-trivial intersection. However, they are distinct sets, as we see below. 

\begin{proposition}$\mathfrak{C}_1^+\neq\mathfrak{C}_1^-$.\end{proposition}
\begin{proof}Corollary 2 of \cite{Bohr02} shows that if $K$ is concordant to a non-trivial strongly quasipositive knot, then $\kappa_+(K)>0$. This implies that $K\notin\mathfrak{C}_1^-$. However, several strongly quasipositive knots are in $\mathfrak{C}_1^+$. For instance, any knot which is a closure of a positive braid (and therefore contained in $\mathcal{P}_0$) is strongly quasipositive. In fact, Rudolph showed that any knot with a positive projection is strongly quasipositive \cite{Rud99}. This shows that all knots with positive projections are in $\mathfrak{C}_1^+ - \mathfrak{C}_1^-$.

Alternatively, Gompf showed that there exist non-trivial knots with kinkiness $(0,n)$, with $n\neq 0$ in \cite{Gom86} (see Section \ref{positivitysection} for a definition of kinkiness). These knots are clearly in $\mathfrak{C}_1^- - \mathfrak{C}_1^+$. \end{proof}

There also exist knots which are is neither $\mathfrak{C}_1^+$ nor $\mathfrak{C}_1^-$, as follows. Recall that $\mathfrak{C}_1=\mathcal{C}$.

\begin{proposition}$\mathfrak{C}_1 \neq \mathfrak{C}_1^+\cup\mathfrak{C}_1^-$.\end{proposition}

\begin{proof}As we saw above, by \cite[Corollary 2]{Bohr02}, any strongly quasipositive knot $K$ has $\kappa_+(K)>0$ and therefore $K\notin\mathfrak{C}_1^-$. However, Rudolph showed in \cite{Rud83,Rud05} that one can construct a strongly quasipositive knot with any given Seifert pairing. As a result, we can find a strongly quasipositive knot $K$ with positive Levine--Tristram signature. By Proposition 1.2 of \cite{CHHo13}, $K\notin\mathcal{P}_0$ and therefore, $K\notin\mathfrak{C}_1^+$. \end{proof}

Clearly, any of the knots guaranteed by the above proposition must have both $\kappa_+(K)$ and $\kappa_-(K)$ non-zero and in fact, this condition characterizes all knots in $\mathfrak{C}_1 - \left(\mathfrak{C}_1^+\cup\mathfrak{C}_1^-\right)$.

\begin{proposition}\label{2neq1}$\mathfrak{C}_2\neq\mathfrak{C}_1$, $\mathfrak{C}_2^+\neq\mathfrak{C}_1^+$ and $\mathfrak{C}_2^-\neq\mathfrak{C}_1^-$.\end{proposition}
\begin{proof}The figure eight knot $4_1$ is contained in both $\mathfrak{C}_1^+$ and $\mathfrak{C}_1^-$ since it can be unknotted by changing a single positive or negative crossing. However, we know that $\text{Arf}(4_1)\neq 0$ and so by Corollary 1, it cannot bound a Casson tower of height two. Therefore, $4_1\notin \mathfrak{C}_2$. Since $\mathfrak{C}_2^\pm\subseteq\mathfrak{C}_2$ the result follows.

Of course, any knot with $\text{Arf}\,(K)=1$ lies in $\mathfrak{C}_1 - \mathfrak{C}_2$ by Corollary 1, since $\mathfrak{C}_1=\mathcal{C}$. Similarly, any knot $K$ with $\text{Arf}(K)=1$ and unknotting number one lies in either $\mathfrak{C}_1^+ - \mathfrak{C}_2^+$ or $\mathfrak{C}_1^- - \mathfrak{C}_2^-$.\end{proof}

The above result shows that while the figure eight knot bounds a kinky disk with a single positive (resp. negative) kink, it cannot be extended to a Casson tower of height two. In fact, by Corollary 1, the figure eight knot does not bound \textit{any} height two Casson tower, regardless of the number (and sign) of kinks at the first stage.

\begin{corollary}$\mathfrak{C}_{2,\,0}^+\equiv\mathfrak{C}_2^+\neq\mathcal{P}_0$. Similarly, $\mathfrak{C}_{2,\,0}^-\equiv\mathfrak{C}_2^-\neq\mathcal{N}_0$.\end{corollary}
\begin{proof}This follows immediately from the previous proposition since $\mathfrak{C}_1^+\subseteq\mathcal{P}_0$ and $\mathfrak{C}_1^-\subseteq\mathcal{N}_0$. \end{proof}

Recall that $T^\pm_n$ denotes the twist knot with $n$ twists, where the superscript denotes the sign of the clasp (see Figure \ref{twistknots}).

\begin{proposition}\label{height2ex}For even $n$, $T^+_n\in\mathfrak{C}^+_{2,\,0}$ and $T^-_n\in\mathfrak{C}^-_{2,\,0}$.\end{proposition}

Notice that by Corollary 1, knots in $\mathfrak{C}^\pm_{2,\,0}$ must have zero Arf invariant. As a result, for odd $n$, $T^\pm_n$ cannot be contained in $\mathfrak{C}^\pm_{2,\,0}$, since $\text{Arf}\,(T^\pm_n)\equiv n\mod 2$.

\begin{proof}[Proof of Proposition \ref{height2ex}]Let $K$ denote $T^\pm_{2k}$ for some $k\in\mathbb{Z}$. The knot $K$ bounds an obvious kinky disk $D_1$ in $B^4$ with a single positive (resp. negative) kink, corresponding to changing one of the two crossings at the clasp. The standard curve, which would need to bound a second stage kinky disk, is an unknot which can be seen as the `core' curve of $K$, shown in Figure \ref{height2exfig1}. Call this curve $\alpha$. As depicted in the figure, $\alpha$ is `mostly' contained in a single slice of $B^4$ (with respect to the radial function). Let this radius be denoted $t_0$. $D_1$ is contained in the region of $B^4$ with radii $\geq t_0$, and as a result, we see that $\alpha$ bounds an \textit{embedded} disk $\widetilde{D_2}$ away from $D_1$, on the side of $B^4$ with radius $< t_0$. If $\widetilde{D_2}$ had the correct framing, it would imply that each $T^\pm_{2k}$ is slice; since this is false, $\widetilde{D_2}$ must have the wrong framing. Note that since $\alpha$ is primarily in the copy of $S^3$ corresponding to radius $t_0$, the push off along the canonical framing is the parallel with the standard zero (Seifert) framing. On the other hand we see that the push off corresponding to $\widetilde{D_2}$ has linking number $2k$ with $\alpha$, i.e.\ the framing is off by $2k$. Our goal for the rest of this proof is to correct the framing of $\widetilde{D_2}$ and obtain an acceptable second-stage kinky disk, $D_2$. Recall that our framing convention is homological, that is, we are seeking a $D_2$ such that for a push off $D_2'$ along the canonical framing, the signed count of intersections between $D_2$ and $D_2'$ should be zero. 

\begin{figure}[t!]
  \begin{center}
  \includegraphics[width=4.5in]{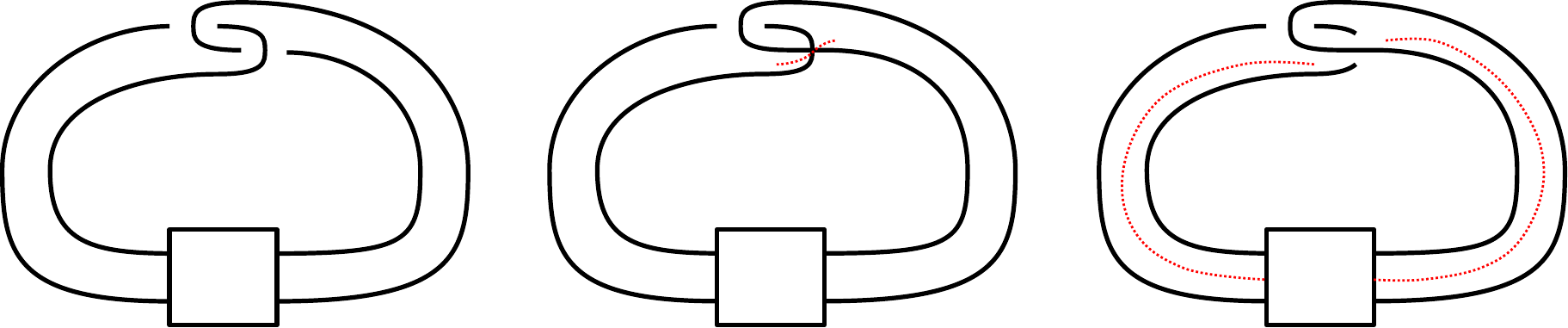}
  \put(-3.91,0.13){$n$}
  \put(-2.34,0.13){$n$}
  \put(-0.77,0.13){$n$}
  \put(-4.2,-0.25){$t>t_0+\epsilon$}
  \put(-2.63,-0.25){$t=t_0+\epsilon$}
  \put(-0.89,-0.25){$t=t_0$}
  \caption{Homotopy showing the base-level kinky disk bounded by $T_n^\pm$. $\alpha$, the standard curve to which the second-level kinky disk should be attached, is shown dotted.}\label{height2exfig1}
  \end{center}
\end{figure}
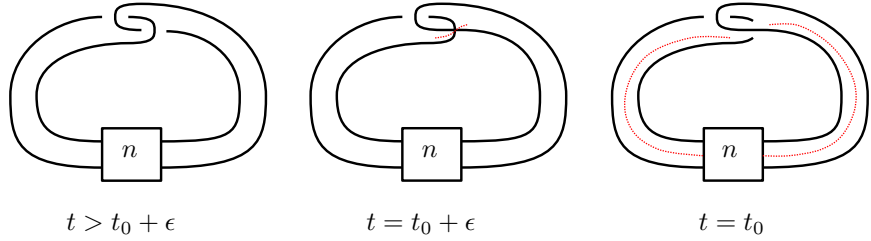

Around the (single) kink in $D_1$, we have a linking torus $T$, which intersects $\widetilde{D_2}$ transversely once. For a precise description of the linking torus at the transverse point of intersection of two planes, see \cite[p. 12]{FreedQ90}. All we will need here is that the meridian and longitude of the linking torus are respectively meridians of the intersecting planes. Therefore, in our case, they are both meridians of $D_1$. 

Assume $T$ is oriented such that $T\cdot\,\widetilde{D}_2 =-1$. Take $k$ parallel (non-intersecting) copies of $T$. We can smooth the intersection between each copy of $T$ and $\widetilde{D_2}$ to obtain a connected surface bounded by $\alpha$. The embedded surface $\Sigma$ thus obtained is homologically $\widetilde{D_2} + kT$. The smoothing process is described in \cite[p. 38]{GomStip99} and can be performed without introducing any self-intersections of $\Sigma$. 

\begin{figure}[t!]
  \begin{center}
  \includegraphics[width=1.8in]{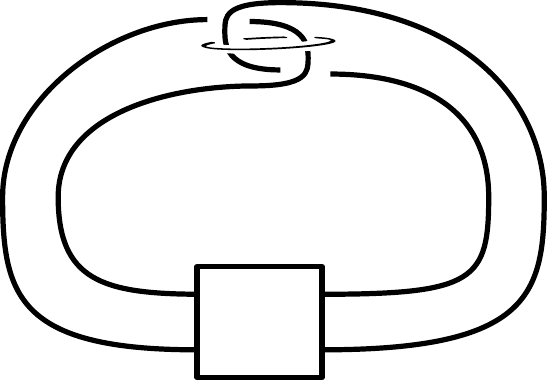}
  \put(-1,0.15){$n$}
  \put(-0.65,1.1){$\eta$}  
  \caption{The curve $\eta$}\label{height2exfig2}
  \end{center}
\end{figure}

The framing of $\Sigma$ (i.e.\ the homological self-intersection number) changes by $2\widetilde{D_2}\cdot\, kT=-2k$. We now have a correctly framed surface of genus $k$ bounded by $\alpha$. We will now use surgery to obtain a kinky disk bounded by $\alpha$. 

\begin{figure}[b!]
  \begin{center}
  \includegraphics[width=4in]{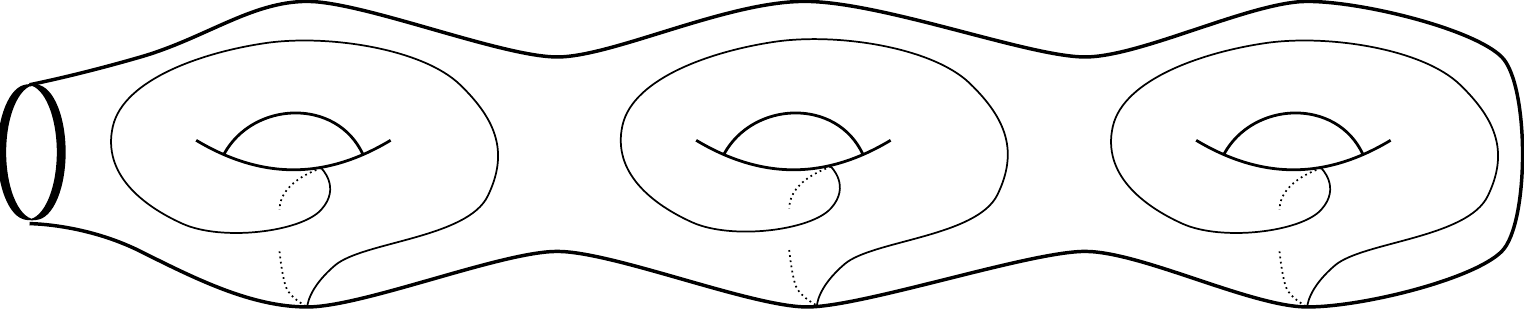}
  \caption{Proof of Proposition \ref{height2ex}.}\label{height2exfig3}
  \end{center}
\end{figure}

Assume $k=1$ for the moment. Then $\Sigma$ is a genus one surface. Consider the $(1,-1)$ curve on $\Sigma$. The meridian and longitude of $\Sigma$ are the same as the meridian and longitude of $T$, and therefore the $(1,-1)$ curve on $\Sigma$ is isotopic to the curve $\eta$ shown in Figure \ref{height2exfig2}, in the exterior of $D_1$. For larger values of $k$ we can find a set of curves, shown (abstractly) in Figure \ref{height2exfig3}, which are each isotopic to $\eta\subseteq S^3$ away from $D_1$. These curves are the images of the $(1,-1)$ curves on $T$ in $\Sigma$---this is easily seen from the construction of $\Sigma$. Surgering along these curves away from $D_1$ would give us a (correctly framed) second stage kinky handle and complete the proof. The resulting disk will have the correct framing since surgery does not change framing (this is because we are using a homological framing and surgery does not change the homology class).

\begin{figure}[t!]
  \begin{center}
  \includegraphics[width=2.5in]{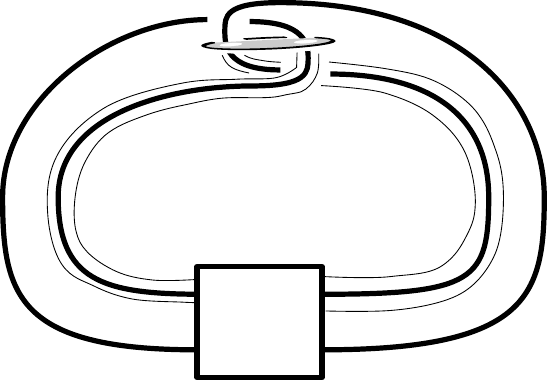}
  \put(-1.35,0.25){$n$}
  \put(-0.92,1.55){$\eta$}  
  \caption{The curve $\eta$ bounds a surface.}\label{height2exfig4}
  \end{center}
\end{figure}

The curve $\eta$ bounds a genus one surface away from $D_1$ as shown in Figure \ref{height2exfig4}. The longitude of this surface is isotopic (away from $D_1$) to the standard curve of $D_1$, namely $\alpha$. We know that $\alpha$ bounds a disk, $\widetilde{D_2}$, away from $D_1$. Surgering using parallel copies of $\widetilde{D_2}$, we see that $\eta$ bounds an immersed disk $\delta$ away from $D_1$. Note that $\delta$ will necessarily intersect $\widetilde{D_2}$ (and therefore $\Sigma$).

We can use $\delta$ to surger $\Sigma$ when $k=1$. For larger values of $k$, we will need multiple parallel copies of $\delta$, which will necessarily intersect one another. However, as long as there are no intersections with $D_1$, we still create a Casson tower of height two as desired\footnote{The author is grateful to Robert Gompf for suggesting a key step in the proof for Proposition \ref{height2ex}.}. \end{proof}

Recall that $\Wh_n(K)$ denotes the $n$--twisted Whitehead double of the knot $K$, where the superscript indicates the type of clasp. By a very similar argument as above, we can show the following.

\begin{proposition}\label{twistedwhitehead}For even $n$ and any knot $K$, $\Wh[+]_n(K)\in\mathfrak{C}_{2,\,0}^+$ and $\Wh[-]_n(K)\in\mathfrak{C}_{2,\,0}^-$.\end{proposition}
\begin{proof}The argument in this case differs from the proof of the previous proposition only in a few details. As before, $\Wh_{2k}(K)$ bounds a first stage kinky disk $D_1$ with a single positive (resp. negative) kink. The standard curve $\alpha$ is no longer an unknot as in the previous case, but has the knot type of $K$. However, any knot $K$ bounds a (correctly framed) kinky disk in the 4--ball, since it can be unknotted by changing crossings (see Example \ref{allc1}). However, the $n$ twists in the Whitehead doubling operator used imply that a regular neighborhood of the na\"{i}ve choice of second stage disk, $\widetilde{D_2}$, is twisted $2k$ times. Fortunately, as before, we can tube with the linking torus at the (single) kink in the first stage disk, and surger repeatedly using copies of $\widetilde{D_2}$ to obtain a Casson tower of height two. The proof is identical to the proof of Proposition \ref{height2ex} apart from the fact that $\widetilde{D_2}$ is no longer embedded. Several new intersections are created as before, but they are all in the second stage kinky disk. \end{proof}

\begin{corollary}$\mathfrak{C}_{2,\,0}^+\neq \mathfrak{C}_{2,\,1}^+$, $\mathfrak{C}_{2,\,0}^-\neq \mathfrak{C}_{2,\,1}^-$, and $\mathfrak{C}_{2,\,0}\neq \mathfrak{C}_{2,\,1}$. \end{corollary}

\begin{proof}The knots $T^+_n$ and $\Wh[+]_n(K)$ are algebraically slice exactly when $n=l(l+1)$ with $l\geq 0$ \cite{CasG86}. (Similarly, knots $T^-_n$ and $\Wh[-]_n(K)$ are algebraically slice exactly when $n=-l(l+1)$ with $l\geq 0$.) This fact, together with Proposition \ref{height2ex}, yields infinitely many knots in $\mathfrak{C}_{2,\,0}^+  -  \mathfrak{C}_{2,\,1}^+$, $\mathfrak{C}_{2,\,0}^-  -  \mathfrak{C}_{2,\,1}^-$, and $\mathfrak{C}_{2,\,0}  -  \mathfrak{C}_{2,\,1}$. This is because, by Corollary 2, knots in $\mathfrak{C}_{2,\,1}^\pm$ or $\mathfrak{C}_{2,\,1}$  must be algebraically slice. \end{proof}

\begin{corollary}$\mathfrak{C}_3^+\neq\mathfrak{C}_2^+$, $\mathfrak{C}_3^-\neq\mathfrak{C}_2^-$, and $\mathfrak{C}_3\neq\mathfrak{C}_2$. \end{corollary}
\begin{proof}Since $\mathfrak{C}_3^\pm\subseteq\mathfrak{C}_{2,\,1}^\pm$ and $\mathfrak{C}_3\subseteq\mathfrak{C}_{2,\,1}$, this follows immediately from the previous corollary.\end{proof}

From the proof of Proposition \ref{twistedwhitehead}, it is tempting to speculate that iterated twisted Whitehead doubles bound arbitrarily high Casson towers, i.e.\ if a knot $K$ bounds a Casson tower of height $p$, $\Wh_n(K)$ bounds a Casson tower of height $p+1$ for any $n$. Unfortunately, this does not follow when $n\neq 0$. In particular, if our wishful thinking were correct, twist knots would bound arbitrarily high Casson towers (since they are twisted doubles of the unknot, which bounds arbitrarily high Casson towers). However, we know this is not true since some twist knots are not algebraically slice and therefore, do not bound Casson towers of height three. 

In the $n=0$ case, we get the following result. 

\begin{proposition}For any knot $K\in\mathfrak{C}_k$, $$\Wh[+]_0(K)\in\mathfrak{C}_{k+1}^+$$ and $$\Wh[-]_0(K)\in\mathfrak{C}_{k+1}^-.$$\end{proposition}

\begin{proof} If a knot $J$ bounds a Casson tower of height $k$, $\Wh_0(J)$ bounds a Casson tower of height $k+1$, with a single kink in the lowest level with sign corresponding to the sign of the clasp of the pattern used. The result follows. \end{proof}

\begin{remark}Note that the above proposition implies that for any knot $K\in\mathfrak{C}_k$, $$\Wh[+]_0( \underbrace{\Wh_0 \cdots \Wh_0}_{n-1 \text{ times}}(K))\in\mathfrak{C}_{n+k}^+$$ and $$\Wh[-]_0( \underbrace{\Wh_0 \cdots \Wh_0}_{n-1 \text{ times}}(K))\in\mathfrak{C}_{n+k}^-.$$\end{remark}

The following is an immediate corollary of the above proposition, Theorem A, Proposition \ref{twistedwhitehead} and Corollary \ref{height3again}.

\begin{corollary}For any even $k$ and any knot $K$, $$\Wh[+]_0(\Wh_k(K))\in\mathfrak{C}_3^+\subseteq\bigcap_{n}\mathfrak{C}_{2,\,n}^+\subseteq\bigcap_{n} \mathcal{P}_n$$  and $$\Wh[-]_0(\Wh_k(K))\in\mathfrak{C}_3^-\subseteq\bigcap_{n}\mathfrak{C}_{2,\,n}^-\subseteq\bigcap_{n} \mathcal{N}_n.$$ \end{corollary}

The above is related to Corollary 3.7 in \cite{CHHo13}, which shows that if $J\in\mathcal{P}_0$ then $\Wh_0(J)$ is in $\bigcap_{n}\mathcal{P}_n$. For any $K$ and $n$, we know that $\Wh[+]_n(K)\in\mathcal{P}_0$, and therefore it was already known that $\Wh[+]_0(\Wh[+]_n(K))\in\bigcap_n\mathcal{P}_n$. However, it is \textbf{not} generally true that $\Wh[-]_n(K)\in\mathcal{P}_0$ for any $K$. For example, $\Wh[-]_0(LHT)\notin\mathcal{P}_0$. 

\section{Generalization to links}\label{links}

The definitions of $\mathfrak{C}_n$, $\mathfrak{C}_n^\pm$, $\mathfrak{C}_{2,\,n}$ and $\mathfrak{C}_{2,\,n}^\pm$ can be naturally generalized to the context of links. Since the connected sum  operation is not well-defined on links, we have to consider the \textit{string link concordance group} of $m$--component string links, denoted $\mathcal{C}(m)$, under the concatenation operation. For $L\in\mathcal{C}(m)$, let $\widehat{L}$ denote the $m$--component link obtained by taking the closure of $L$. 

\begin{defn_1'}An $m$--component string link link $L$ is said to be in $\mathfrak{C}_n(m)$ if $\widehat{L}_i$, the components of $\widehat{L}$, bound disjoint Casson towers of height $n$. \end{defn_1'}

\begin{defn_2'}An $m$--component string link $L$ is said to be in $\mathfrak{C}_{2,\,n}(m)$ if $\widehat{L}_i$, the components of $\widehat{L}$, bound disjoint Casson towers $T_i$ of height two such that each member of a standard set of curves for each $T_i$ is in $\pi_1(B^4  -  \sqcup_i\, T_i)^{(n)}$. \end{defn_2'}

\begin{defn_3'}An $m$--component string link link $L$ is said to be in $\mathfrak{C}_n^+(m)$ (resp. $\mathfrak{C}_n^-(m)$) if $\widehat{L}_i$, the components of $\widehat{L}$, bound disjoint Casson towers of height $n$ such that all the kinks in the first stage kinky disks are positive (resp. negative). \end{defn_3'}

\begin{defn_4'}An $m$--component string link $L$ is said to be in $\mathfrak{C}_{2,\,n}^+(m)$ (resp. $\mathfrak{C}_{2,\,n}^-(m)$) if $\widehat{L}_i$, the components of $\widehat{L}$, bound disjoint Casson towers $T_i$ of height two such that all the kinks in the first stage kinky disks are positive (resp. negative) and each member of a standard set of curves for each $T_i$ is in $\pi_1(B^4  -  \sqcup_i\, T_i)^{(n)}$. \end{defn_4'}

There are similar definitions for the grope filtrations $\mathcal{G}_{n}(m)$ and $\mathcal{G}_{2,\,n}(m)$, and the $n$--solvable filtration $\mathcal{F}_n(m)$ for $\mathcal{C}(m)$ which we omit for the sake of brevity---they are identical to the definitions in the case of knots, except that the components of the link are required to bound \textit{disjoint} disks in 4--manifolds of the relevant flavor. Positive links, i.e.\ links in $\mathcal{P}_0(m)$, have been studied by Cochran--Tweedy in \cite{CTwee13}.

Since all of our arguments in Chapter \ref{tower_results} take place within Casson towers, the results generalize easily to links. Therefore, we obtain the following theorem. 

\begin{thm_A'}For any $n\geq 0$, and $m\geq 1$, 
\begin{itemize}
\item[(i)]$\mathfrak{C}_{n+2}(m)\subseteq\mathcal{G}_{n+2}(m)\subseteq\mathcal{F}_n(m)$
\item[(ii)]$\mathfrak{C}_{2,\,n}(m)\subseteq\mathcal{G}_{2,\,n}(m)\subseteq\mathcal{F}_n(m)$.
\item[(iii)]$\mathfrak{C}_{n+2}^+(m) \subseteq \mathfrak{C}_{2,\,n}^+(m) \subseteq \mathcal{P}_n(m)$
\item[(iv)]$\mathfrak{C}_{n+2}^-(m) \subseteq \mathfrak{C}_{2,\,n}^-(m)\subseteq \mathcal{N}_n(m)$
\end{itemize}
\end{thm_A'}

Note that $\mathfrak{C}_1^+(m)\subseteq\mathcal{P}_0(m)$ and $\mathfrak{C}_1^-(m)\subseteq\mathcal{N}_0(m)$, even in the case of links. Using a near-identical proof to that of Theorem B, we obtain the following.

\begin{thm_B'}For any $m$--component string link $L$, the following statements are equivalent. 
\begin{itemize}
\item[(a)] $L\in \mathfrak{C}_1^+(m)$ (resp. $\mathfrak{C}_1^-(m)$)
\item[(b)] $\widehat{L}$ is concordant to a link each of whose components is a fusion knot of a split collection of positive (resp. negative) Hopf links
\item[(c)] $\widehat{L}$ is concordant to a link each of whose components can be changed to a ribbon knot by changing only positive (resp. negative) crossings (within the same component).
\end{itemize}
\end{thm_B'}

Recall that any knot $K$ lies in $\mathfrak{C}_1$ since it can be unknotted by changing some number of crossings. However, it is not true that every $m$--component link lies in $\mathfrak{C}_1(m)$ as we see below.

Recall that two links are link homotopic if we can go from one to the other via a deformation where each component may intersect itself but distinct components must remain disjoint. 

\begin{proposition}\label{nullhomotopic}If an $m$--component string link $L$ lies in $\mathfrak{C}_1(m)$, then $\widehat{L}$ is link homotopic to the $m$--component unlink and, in particular, the pairwise linking numbers of $\widehat{L}$ are zero.\end{proposition}

\begin{proof} Since $L\in\mathfrak{C}_1(m)$, the components of $\widehat{L}$ bound disjoint immersed disks in $B^4$. By following the proof of Theorem B$'$, we see that $\widehat{L}$ is concordant to a link $\widehat{M}$ which can be changed to a ribbon link by changing some number of crossings, i.e.\ $\widehat{M}$ is link homotopic to a ribbon link. However, we know from \cite{Giffen79, Gold79} that link concordance implies link homotopy. Since $\widehat{L}$ is concordant to $\widehat{M}$ and any $m$--component ribbon link is concordant to the $m$--component unlink, we have that $\widehat{L}$ is link homotopic to $\widehat{M}$ which is link homotopic to a ribbon link which is link homotopic to the $m$--component unlink. 

The linking number between two simple closed curves in $S^3$ can be computed as the signed intersection number between 2--chains bounded by them in $B^4$ \cite[p. 136]{Ro90}. Since the components of $\widehat{L}$ bound disjoint 2--chains (in particular, immersed disks) in $B^4$ all the pairwise linking numbers are zero. Alternatively, recall that pairwise linking numbers are particular cases of Milnor's invariants with distinct indices, which are invariants of link homotopy. The fact that pairwise linking numbers vanish then follows from the fact that $L$ is link homotopic to the unlink\footnote{This alternative proof was pointed out by an anonymous referee}. \end{proof}

As in Corollary \ref{gropeintersection}, we obtain the following result. 

\begin{cor_33'}Let $\mathcal{T}(m)$ denote the set of all topologically slice string links with $m$ components. Then, for any $m\geq 1$, $$\mathcal{T}(m)\subseteq\bigcap_{n=1}^\infty\mathcal{G}_n(m).$$ \end{cor_33'}

As we mentioned in Chapter \ref{defns}, the groups $\mathcal{G}_{2,\,n}(m)$ have not appeared previously in the literature, but several results relating to the grope filtration carry over easily. In the case of links, this can be seen in context of $k$--cobordism of links (\cite[Definition 9.1]{C90}\cite{Sat84}) as follows. We reference the corresponding results from \cite{Cot12} regarding the grope filtration below since our proofs are essentially the same.

\begin{proposition}[Proposition 6.4 of \cite{Cot12}]If $L\in\mathcal{G}_{2,\,n}(m)$ then $L$ is $2^{n+1}$--cobordant to a slice link, i.e.\ $L$ is null $2^{n+1}$--cobordant. \end{proposition}
\begin{proof}The proof is essentially identical to the proof of Proposition 6.4 of \cite{Cot12}, which says that if $L\in\mathcal{G}_{n+2}(m)$ then $L$ is null $2^{n+1}$--cobordant. Her proof only uses the fact that each member of a symplectic basis for the first stage surfaces (call them $\Sigma_i$) of the gropes lies in $\pi_1(B^4-\sqcup_i\,\Sigma_i)$, which clearly still holds for a link in $\mathcal{G}_{2,\,n}(m)$.\end{proof}

Corollary 2.2 of \cite{Lin91} states that if a link $L$ is null $k$--cobordant, then Milnor's $\overline{\mu}$--invariants of $L$ with length less than or equal to $2k$ vanish. Therefore, we obtain the following corollary. 

\begin{corollary}[Corollary 6.6 of \cite{Cot12}]\label{mubarvanishes}If $L\in\mathcal{G}_{2,\,n}(m)$, then $\overline{\mu}_L(I)=0$ for $\lvert I\rvert \leq 2^{n+2}$.\end{corollary}

Since $\mathfrak{C}_{2,\,n}(m)\subseteq \mathcal{G}_{2,\,n}(m)$ for all $n$ and $m$, we also obtain the following. 

\begin{corollary}[Corollary 6.6 of \cite{Cot12}]\label{mubarvanishes2}If $L\in\mathcal{C}_{2,\,n}(m)$, then $\overline{\mu}_L(I)=0$ for $\lvert I\rvert \leq 2^{n+2}$.\end{corollary}

\begin{proposition}\label{linkdiffs}For $m\geq 2^{n+2}$ and $n\geq0$,
\begin{enumerate}
\item[(a)] $\mathbb{Z}\subseteq \mathcal{F}_{n}(m)\slash\mathcal{G}_{2,\,n}(m)$,
\item[(b)] $\mathbb{N}\subseteq \mathcal{P}_{n}(m)\slash\mathcal{G}_{2,\,n}(m)$,
\item[(c)] $\mathbb{N}\subseteq \mathcal{N}_{n}(m)\slash\mathcal{G}_{2,\,n}(m)$.
\end{enumerate}
\end{proposition}
\begin{proof}The proof of $(a)$ is very closely related to Otto's proof of \cite[Corollary 6.8]{Cot12} in light of Corollary \ref{mubarvanishes}. Here is a short sketch. Let $H$ denote the positive Hopf link, and $BD^i(H)$ its $i^\text{th}$ iterated Bing double (where each component of a link gets Bing doubled at each step). Otto shows that $BD^{n+1}(H)\in\mathcal{F}_n(2^{n+2})$ for each $n$. Work of Cochran \cite[Theorem 8.1]{C90} then shows that $\overline{\mu}_{BD^{n+1}(H)}(I)=1$ for some $I$ of length $2^{n+2}$ with distinct indices (note that $BD^{n+1}(L)$ has $2^{n+2}$ components) and additionally, all $\overline{\mu}$--invariants of smaller length vanish. Corollary \ref{mubarvanishes} shows that $BD^{n+1}(H)\in\mathcal{F}_n(2^{n+2})\slash\mathcal{G}_{2,\,n}(2^{n+2})$ for each $n$. Since the first non-zero $\overline{\mu}$--invariant is additive under concatenation of string links \cite[Theorem 8.13]{C90}\cite{O89} and each $\mathcal{F}_n(2^{n+2})$ is a subgroup of $\mathcal{C}(2^{n+2})$, we see that $BD^{n+1}(H)$ generates an infinite cyclic subgroup of $\mathcal{F}_{n}(2^{n+2})\slash\mathcal{G}_{2,\,n}(2^{n+2})$. By adding unknotted and unlinked components to $BD^{n+1}(H)$ away from all of the other components, we see that $\mathbb{Z}\subseteq \mathcal{F}_{n}(m)\slash\mathcal{G}_{2,\,n}(m)$ for all $m\geq 2^{n+2}$ and $n\geq0$.

\begin{figure}[t!]
  \begin{center}
  \includegraphics[width=2.75in]{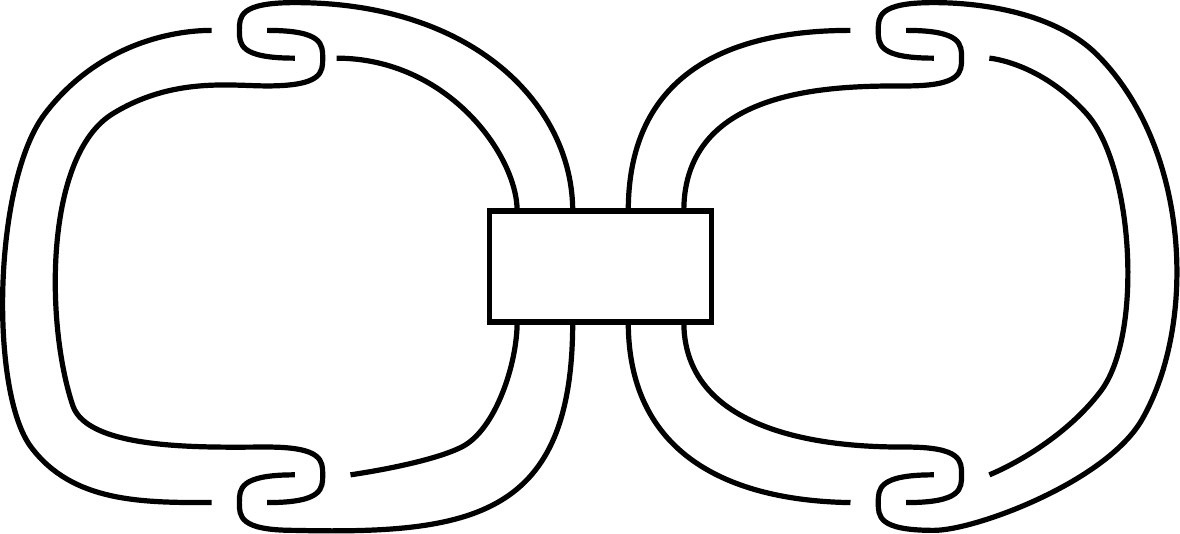}
    \put(-1.5,0.57){$-1$}
  \caption[A link in $\mathcal{P}_0(4)-\mathfrak{C}_1^\pm(4)$]{The above link is in $\mathcal{P}_0(4)$ but not in either $\mathfrak{C}_1^\pm(4)$. The strands going through the box marked with $-1$ are given a full negative twist relative to one another. }\label{linkex}
  \end{center}
\end{figure}

We give the proof for $(b)$; taking concordance inverses of these examples will complete the proof for $(c)$. Consider the link $L$ given in Figure \ref{linkex}. In \cite[Example 4.13]{CTwee13}, it is shown that $L\in\mathcal{P}_0(4)$ since it is obtained from an unlink by adding a so-called `generalized positive crossing'. (This operation was defined in \cite{CTwee13} and consists of adding a full \textit{negative} twist to a collection of strands where each link component is represented algebraically zero times, i.e.\ this corresponds exactly to the box marked with $-1$ in Figure \ref{linkex}.) By \cite[Lemma 3.7]{ChaPow14}, $BD^{n}(L)\in\mathcal{P}_n(2^{n+2})$. However, $BD^{n}(L)$ is link homotopic to $BD^{n+1}(H)$. Since $\overline{\mu}$--invariants with distinct indices are invariants of link homotopy, we see that $BD^{n}(L)\in \mathcal{P}_{n}(2^{n+2})\slash\mathcal{G}_{2,\,n}(2^{n+2})$ by Corollary \ref{mubarvanishes}. Unlike $\mathcal{F}_n(2^{n+2})$, $\mathcal{P}_n(2^{n+2})$ is merely a submonoid of $\mathcal{C}(2^{n+2})$ and therefore, we only get $\mathbb{N}\subseteq \mathcal{P}_{n}(2^{n+2})\slash\mathcal{G}_{2,\,n}(2^{n+2})$. By adding unknotted and unlinked components to $BD^{n}(L)$, we see that $\mathbb{N}\subseteq \mathcal{P}_{n}(m)\slash\mathcal{G}_{2,\,n}(m)$ for all $m\geq 2^{n+2}$ and $n\geq0$.
\end{proof}

In the case of links we also obtain the following additional results, which we are currently unable to prove in the case of knots. 

\begin{proposition}$\mathcal{P}_0(m)\neq\mathfrak{C}_1^+(m)$ and $\mathcal{N}_0(m)\neq\mathfrak{C}_1^-(m)$ for $m\geq 4$. \end{proposition}
\begin{proof}Links demonstrating this inequality may be found in \cite[Example 4.13]{CTwee13}. The link $L$ shown in Figure \ref{linkex} is link homotopic to the Bing double of a Hopf link (this is easily seen by drawing a picture of both links; recall that the box in Figure \ref{linkex} containing a `-1' indicates a full negative twist of all the strands passing through it) and therefore, has non-zero $\overline{\mu}(1234)$. This implies that $L$ is not  link homotopic to the unlink, and therefore, by Proposition \ref{nullhomotopic} is not in $\mathfrak{C}_1^+(m)$. However, we see that $L\in\mathcal{P}_0(m)$ in \cite{CTwee13}.

The mirror image of the link in Figure \ref{linkex} is in $\mathcal{N}_0(m) - \mathfrak{C}_1^-(m)$.\end{proof}

We can actually do better by following the proof of Proposition \ref{linkdiffs}, as follows.

\begin{proposition}For $m\geq 2^{n+2}$ and $n\geq0$,
\begin{enumerate}
\item[(a)] $\mathbb{Z}\subseteq \mathcal{F}_{n}(m)\slash\mathfrak{C}_1(m)$,
\item[(b)] $\mathbb{N}\subseteq \mathcal{P}_{n}(m)\slash\mathfrak{C}_1^+(m)$,
\item[(c)] $\mathbb{N}\subseteq \mathcal{N}_{n}(m)\slash\mathfrak{C}_1^-(m)$.
\end{enumerate}
\end{proposition}
\begin{proof}In the proof of Proposition \ref{linkdiffs}, we demonstrated the existence of links which are in $\mathcal{F}_n(m)$ (resp. $\mathcal{P}_n(m)$, $\mathcal{N}_n(m)$) and have a non-zero $\overline{\mu}$--invariant with distinct indices. These links are therefore not link homotopic to the unlink, and as a result, by Proposition \ref{nullhomotopic} are not contained in $\mathfrak{C}_1(m)$ (resp. $\mathfrak{C}^+_1(m)$, $\mathfrak{C}_1^-(m)$). \end{proof}


%
%
%
\bibliographystyle{gtart}
\bibliography{knotbib}

%

\end{document}